\documentclass[journal]{IEEEtran}
\usepackage{defieee}
\def\revised#1{{\color{black}#1}}
\def\deleted#1{{\color{black}#1}}
\def\secondrevised#1{{\color{black}#1}}
\usepackage{todonotes}
\allowdisplaybreaks
\begin{document}

\title{Stabilizing Stochastic Predictive Control under Bernoulli Dropouts}
\author{Prabhat\ K.\ Mishra,~\IEEEmembership{Student Member,~IEEE,}
        Debasish~Chatterjee,~\IEEEmembership{}%
        and~Daniel~E.~Quevedo,~\IEEEmembership{Senior Member,~IEEE}
\thanks{Prabhat K. Mishra and Debasish Chatterjee are with Systems and Control Engineering, Indian Institute of Technology Bombay, Mumbai, India.  
        \tt{ prabhat@sc.iitb.ac.in, dchatter@iitb.ac.in 
        }}
\thanks{Daniel E. Quevedo is with Department of Electrical Engineering (EIM-E),
 Paderborn University,   Germany. \tt{dquevedo}@ieee.org}        
\thanks{D.\  Chatterjee was supported in part by the grant 12IRCCSG005 from IRCC, IIT Bombay.}
\thanks{The authors thank Peter Hokayem for helpful discussions.}}%

{}
\maketitle
\begin{abstract}
This article presents \revised{ tractable and recursively feasible optimization-based controllers for stochastic linear systems with bounded controls.} The stochastic noise in the plant is assumed to be additive, zero mean \revised{and fourth moment bounded}, \revised{and the control values transmitted over an erasure channel. Three different transmission protocols are proposed having different requirements on the storage and computational facilities available at the actuator}. \revised{We optimize a suitable stochastic cost function accounting for the effects of both the stochastic noise and the packet dropouts over affine saturated disturbance feedback policies}. The proposed controllers ensure mean square boundedness of the states in closed-loop for \revised{all positive values of control bounds and any non-zero probability of successful transmission over a noisy control channel. 
}
\end{abstract}

\begin{IEEEkeywords}
model predictive control, stochastic control, erasure channel, packet dropouts.
\end{IEEEkeywords}

%
\IEEEpeerreviewmaketitle

\section{Introduction}

\IEEEPARstart{T}{he} availability of fast computing machines has spurred the growth of control techniques involving algorithmic selection of actions that minimize some performance objective. Receding horizon predictive control, which is based on the idea of algorithmic selection of control actions, has evolved over the years into one of the most useful control synthesis techniques currently available to a control engineer. Initial forays into stochastic versions of receding horizon techniques were made in the operations research community, with inventory and manufacturing systems as the key application areas. Since then there has been a steady growth of stochastic receding horizon control in the domain of control systems, with current application areas ranging in financial engineering, industrial electronics, power systems, process control, etc. See e.g., the recent surveys \cite{ref:May-14, mesbah_16_survey} for the current state-of-the-art and \cite{ref:Grune-11, ref:rawlings-09} for book-length treatments. 


\secondrevised{
As emphasized in \cite{ref:May-14}, stochastic receding horizon control still lacks a comprehensive and unified treatment. In particular, the treatment of potentially unbounded noise appears to be sketchy, and part of the reason for this lacuna is the complete absence of a final constraint set that can be made positively invariant. Consequently, the conventional tools developed in deterministic and robust receding horizon control, e.g., in \cite{ref:Bemporad-99, ref:Rossiter-98, ref:Kerrigan-04, ref:Maciejowski-09, ref:Marruedo-02, ref:Bayer-13, ref:rakovic-05, ref:rakovic-12, ref:maeder-09, ref:quevedo-04} do not carry over; see \cite[\S3]{ref:May-14} for a detailed discussion.} The literature is even more sketchy when it comes to receding horizon predictive control over networks that take into account the stochastic effects of such communication networks. This is the precise area where the contributions of the present article lie.

This article is concerned with stochastic predictive control under an unreliable control channel. \revised{Numerous contemporary applications involve control of systems over noisy control channels. For instance, cloud-aided vehicle control systems \cite{li2014cloud, li2015h}, where the control values are computed on remote servers and transmitted to the vehicles over wireless channels. Disturbances enter into the plant dynamics due to, for example, GPS localization errors and transmissions over noisy wireless channels. The interference and the fading effects in the noisy control channel lead to packet delays and dropouts. To avoid book-keeping and for sheer simplicity, in this article we consider delayed information as lost information.\footnote{ The problem of packet delays in the control channel will be studied separately, e.g., by building on models such as in \cite{quevedo_jurado_TAC}.}} The plant is assumed to be a noisy linear controlled system, with the plant noise being a stochastic process entering the dynamics additively. We do \emph{not} assume that the noise is bounded. Notice that even without control channel dropouts, since the noise is not bounded, a conventional terminal constraint argument \revised{along the lines of} \cite{ref:mayne-00} to ensure closed loop stability turns out to be impossible to execute. Since the underlying optimal control problems are employed in a receding horizon fashion, and each of them is a finite horizon optimal control problem, stability of the closed loop system under the resulting receding horizon control is by no means obvious. In the deterministic setting, one selects the cost-per-stage function and the final cost function in a certain way (based on Lyapunov inequalities) to ensure stability of the closed-loop system \cite{ref:KeeGil-88}. \revised{Although such a Lyapunov based argument is possible even in the stochastic case \cite{chatterjee-15}, it is not easily numerically tractable}. To deal with stochastic systems, we shall impose a constraint in the underlying optimal control problem such that it is (globally) feasible and the closed-loop system is stable in a precise sense. Such constraint embedding for stability has appeared before in the literature concerning deterministic model predictive control in \cite{ref:mayne-00} and stochastic model predictive control in \cite{ref:ChaRamHokLyg-12, bernardini2010model}. The control channel is modelled as an erasure channel with dropouts occurring according to an i.i.d. Bernoulli random process. We assume that all communication occurs via a protocol with reliable acknowledgements of receipts. One of our primary goals is to ensure good stability properties in closed-loop operation. To this end we select the notion of mean-square boundedness of the closed-loop process as a desirable property, and we demonstrate how to ensure mean-square boundedness under mild assumptions on the noise. To our knowledge this is the first rigorous treatment of stochastic model predictive control with bounded controls in the presence of control channel dropouts, and where stability of closed-loop process is ensured.

\par Control under packet dropouts is extensively studied within the framework of sequence based control \cite{Dolgov2013, Fischer2013, Hekler2012} and packetized predictive control (PPC) \cite{Quevedo2011, Quevedo-12}. In sequence based control, packet dropouts and communication delays are assumed to be present in both \revised{the control channel} and \revised{the sensor channel}, but control bounds are not incorporated in the problem formulations. Hard bounds on control actions are omnipresent in practical applications; therefore it is of paramount importance that control strategies take into account these bounds at the synthesis stage. PPC handles the constrained control problems by solving a deterministic optimization program over open loop input sequences at each time step. It was demonstrated in \cite[\S 2.4]{kumar1986stochastic} that for stochastic systems, the optimization over feedback policies is in general superior to open-loop controls in the sense that the cost in the former case is lower than in the latter case. However, the set of decision variables is non-convex for general state feedback policies \cite{goulart-06}. When using disturbance feedback parametrization, an admissible feedback policy can be obtained by solving a tractable convex optimization program \cite{goulart-06}. Disturbance feedback parametrization followed by saturation of feedback is used in \cite{ref:HokChaRamChaLyg-10} to satisfy the constraints on actions in the presence of potentially unbounded disturbance. This approach was followed in \cite{ref:amin-10} in the presence of i.i.d. packet dropouts, but mean square bounds of the closed-loop states could be proved only under sufficiently large control authority. Our results extend the main results of \cite{ref:HokChaRamChaLyg-10} and \cite{ref:amin-10} by removing the lower bound on the control authority in the presence of i.i.d. packet dropouts. \revised{The recent works \cite{Quevedo2011, Quevedo-12} have considered predictive control under stochastically modelled packet dropouts and additive stochastic noise with unbounded support, but relied on deterministic cost functions for computational tractability. Stochastic receding horizon control with stochastic cost functions accounting for the dropouts and additive noise has appeared only in our recent conference contributions \cite{ref:quevedo-15, prabhat2016} in its initial stage of the development.} In addition, we demonstrated in \cite{ref:quevedo-15} via numerical experiments that our approach employs less actuator energy than PPC \cite{Quevedo2011, Quevedo-12}. 
 
\par The main contributions of this article are as follows: \revised{ 
\begin{itemize}[leftmargin = *]
\item We give tractable and recursively feasible solutions to stochastic predictive control for linear systems in the presence of an unreliable control channel, possibly unbounded additive noise, and hard bounds on the control actions. 
\item We present three transmission protocols depending upon the availability of storage and computation facilities at the actuator, and systematically show the formulation of numerically tractable optimization programs for receding horizon control.   
\item We introduce stability constraints to ensure mean square boundedness of closed-loop states for any positive bound on the control actions and for any non-zero probability of successful transmission over a noisy control channel.
\end{itemize}
}
\par The article exposes as follows: In \secref{s:problem setup} we establish the notation and definitions of the plant and its properties. In \secref{s:main results} we provide the main results and discuss stability issue in 
\secref{s:stability}. We present numerical experiments in \secref{s:simulation} and conclude in \secref{s:conclusion}. The proofs of our results are documented in a consolidated manner in the Appendix.
\par Our notations are standard. We let $\R$ denote the real numbers, \revised{$\R_{\geq 0}$ denote the non-negative reals,} and \(\Nz\) denote the natural numbers. \revised{Intervals on the real line are denoted by $]a,b[$, $]a,b]$, etc., where $]a,b[ \Let \{z \in \R \mid a < z < b \}$ and $]a,b] \Let \{z \in \R \mid a < z \leq b \}$.} For any sequence \((s_n)_{n\in\Nz}\) taking values in some Euclidean space, we denote by \(s_{n:k}\) the vector \(\pmat{s_n\transp & s_{n+1}\transp & \cdots & s_{n+k-1}\transp}\transp\), \(k\in\Nz\). The standard Kronecker and Schur products of matrices are denoted by \(\otimes\) and \(\odot\), respectively. For any \(z\in\R\) we let \(z_+\) and \(z_-\) denote the positive and the negative parts \(\max\{z, 0\}\) and \(\max\{-z, 0\}\) of \(z\), respectively. The notation \(\EE_z[\cdot]\) stands for the conditional expectation with given \(z\). For a given vector \(V\) its $i^{th}$ component is denoted by $V^{(i)}$. Similarly, $M^{(i)}$ denotes the $i^{th}$ row of a given matrix $M$. The column vector of all $1$'s of length $k$ is denoted by $\ones_{1:k}$. For a matrix $M$ the quantity $\sigma_1(M)$ denotes its largest singular value \secondrevised{and $M^+$ its Moore Penrose pseudo inverse \cite[\S 6.1]{bernstein2009matrix}}. \revised{Inner products on Euclidean spaces are denoted by $\inprod{v}{w} \Let v \transp w$.} We let \(I_d\) denote the \(d\times d\) identity matrix. The \( r\times q\) matrix with all entries equal to 0 is denoted by \(0_{r\times q} \) or simply by $0$ when the dimension of the matrix is clear from the context.

	\section{Problem Setup}
	\label{s:problem setup}

	\subsection{Dynamics and objective function}
	\label{s:problem:dynamics and cost}
		Consider a linear time-invariant control system with additive process noise and controlled over \revised{an unreliable} channel given by
		\begin{equation}
		\label{e:system}
			\st_{t+1} = \A \st_t + \B \control^a_t + \wnoise_t,\quad \st_0 = \stinit,
		\end{equation}
		where 
		\begin{enumerate}[label=(\eqref{e:system}-\alph*), leftmargin=*, widest=b]
			\item \(\st_t \in\R^d\), \(A\in\R^{d\times d}, B\in\R^{d\times m}\) are given matrices, \(\stinit\in\R^d\) is a given vector. 
			The control \(\control^a_t\) is the available control at the actuator end at time \(t\) after passing through the erasure control channel. \deleted{The control is constrained to} take values in the set 
				\begin{equation}\label{e:controlset}
				 \controlset\Let\{v\in\R^m\mid \norm{v}_\infty \le U_{\max}\}, \quad \text{ for all } t,
				\end{equation}
				where $U_{\max} > 0$ is a given constraint,
			\item \((\wnoise_t)_{t\in\Nz}\) is a sequence of i.i.d. zero mean and fourth moment bounded \ random vectors, and
			\item at each time \(t\) the state \(\st_t\) is measured perfectly.
		\end{enumerate}
\par	A control policy \(\pi\) is a sequence \((\pi_0, \ldots, \pi_t, \ldots)\) of Borel measurable maps \(\pi_t:\R^d\lra\controlset\). Policies of finite length \((\pi_t, \pi_{t+1}, \ldots, \pi_{t+N-1})\) for some \(N\in\Nz\) will be denoted by \(\pi_{t:N}\) in the sequel.
The control $\control_t^a$ available at the actuator end at time $t$ depends on the transmitted parameters of the control policy and the dropouts \((\cnoise_t)_{t\in\Nz}\) in the control channel. \((\cnoise_t)_{t\in\Nz}\) is a sequence of i.i.d. Bernoulli \(\{0, 1\}\) random variables with $\EE[\cnoise_t] = p \in \;]0,1]$ and the sequence \((\cnoise_t)_{t\in\Nz}\) is independent of \((\wnoise_t)_{t\in\Nz}\). \revised{Therefore, $\cnoise_t = 0$ refers to packet dropout and $\cnoise_t = 1$ refers to successful transmission.}
\begin{figure}[t]		
\centering
\begin{adjustbox}{width=\columnwidth}
\input{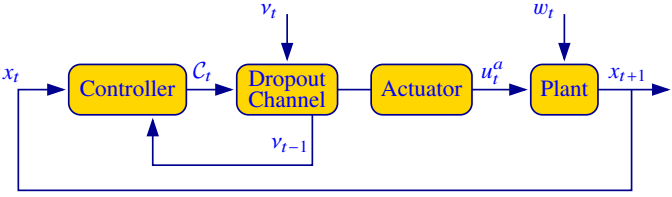}
\end{adjustbox}
\caption{An unreliable communication channel is situated between the controller and the actuator. The parameters of control policy transmitted by the controller depend on the transmission protocols. Acknowledgements of the successful transmission are causally available to the controller.}
 \vspace{-1.3em} 
\label{Fig:blockdia}
\end{figure} 
\par	Let symmetric and non-negative definite matrices \(Q, Q_f\in\R^{d\times d}\) and a symmetric and positive definite matrix \(R\in\R^{m\times m}\) be given, and define the cost-per-stage function $\costps:\R^d\times\controlset\lra \R_{\geq 0} $ and the final cost function \(\costfinal:\R^d\lra \R_{\geq 0} \) by \[ \costps(z, v) \Let \inprod{z}{Qz} + \inprod{v}{Rv} \text{ and } \costfinal(z) = \inprod{z}{Q_f z},\] respectively. Fix an optimization horizon \(N\in\Nz\) and consider the objective function at time \(t\) given the state \(\st_t\):
		\begin{equation}
		\label{e:objfn}
			V_t \Let \EE_{\st_t}\biggl[\sum_{k=0}^{N-1} \costps(\st_{t+k}, \control^a_{t+k}) + \costfinal(\st_{t+N}) \biggr].
		\end{equation}
		The cost function \(V_t\) therefore, considers the control effort that occurs at the actuator end, not just the computed control. At each time instant \(t\) we are interested in minimizing the objective function \(V_t\) over the class of causal history-dependent feedback strategies \(\Pi\) defined by 
		\[
			\control_{t+\ell} = \pi_{t+\ell}(\st_t, \st_{t+1}, \ldots, \st_{t+\ell})
		\]
		while satisfying \(\control_t\in\controlset\) for each \(t\).  

		The \emph{receding horizon control strategy} for a given recalculation interval \(N_r \in\{1, \ldots, N\}\) and time \(t\) consists of successive applications of the following steps:
		\begin{enumerate}[label=(\roman*), leftmargin=*, widest=iii]
			\item measure the state \(\st_t\) and determine an admissible optimal feedback policy \(\pi^\star_{t:N}\in\Pi\) that minimizes \(V_t\) at time \(t\),
			\item apply the first \(N_r \) elements \(\pi^\star_{t:N_r-1}\) of this policy,
			\item increase \(t\) to \(t+ N_r \) and return to step (i).
		\end{enumerate}
		In common parlance, the case of \(N_r = 1\) is known as \emph{model predictive control}, and that of \(N_r = N\) is known as \emph{rolling horizon control} \cite{ref:ChaHokLyg-11, ref:quevedo-03, ref:aboudolas-10, ref:garcial-89}. In this current work we consider the recalculation interval $N_r$ equal to the reachability index $\reachindex$ of the system to satisfy certain drift conditions \eqref{e:drift1} and \eqref{e:drift2} for stability. \revised{Therefore, the optimization horizon $N$ should be not smaller than $\reachindex$.}

		The states, controls and noise over one horizon \(N\) admit the following description under \textsl{unreliable control channel}, for situations where at each time \(t\), solely the  control action \(\control_t \) is sent over the communication channel:
		\begin{equation}
		\label{e:augmented}
			\st_{t:N+1} = \calA \st_t + \calB \control^a_{t:N} + \calD \wnoise_{t:N},
		\end{equation}
		where
		\[
			\control^a_{t:N} \Let \pmat{\cnoise_t\ones_{1:m}\\\cnoise_{t+1}\ones_{1:m}\\\vdots \\\cnoise_{t+N-1}\ones_{1:m}}\odot\pmat{\control_t\\\control_{t+1}\\\vdots\\\control_{t+N-1}} = \pmat{\nu_t \control_t\\ \nu_{t+1} \control_{t+1}\\\vdots\\ \nu_{t+N-1} \control_{t+N-1}},\;
		\]
		\(\calA \Let \pmat{I_d\\ \A\\ \vdots \\ \A^N}\), \(\calB \Let \pmat{0_{d\times m} & 0_{d\times m} & \cdots & 0_{d\times m}\\ \B & 0_{d\times m} & \cdots & 0_{d\times m}\\ \vdots & \vdots & \ddots & \vdots\\ \A^{N-1}\B & \A^{N-2}\B & \dots & \B}\), and \(\calD \Let \pmat{0_{d\times d} & 0_{d\times d} & \cdots & 0_{d\times d}\\ I_d & 0_{d\times d} & \cdots & 0_{d\times d}\\ \vdots & \vdots & \ddots & \vdots\\ \A^{N-1} & \A^{N-2} & \dots & I_{d}}\). \revised{Notice that the expression of $\control_{t:N}^a$ depends on how the stacked control vector $\control_{t:N}$ is affected by the channel effects.} We define two block diagonal matrices 
		\[ \calQ \Let \blkdiag\bigl(\overset{N\text{ times}}{\overbrace{Q, \cdots, Q}}, Q_f\bigr) \text{ and } \calR \Let \blkdiag(\overset{N \text{ times}}{\overbrace{R, \cdots, R}}) \] derived from the given matrices \(Q, Q_f\) and \(R\). 
		The compact notation above allows us to state the following optimal control problem underlying the receding horizon control technique:
		\begin{equation}\label{e:opt problem}
		\begin{aligned}
			& \minimize_{\pi_{t:N}}	&& \EE_{\st_t}\bigl[ \inprod{\st_{t:N+1}}{\calQ \st_{t:N+1}} + \inprod{\control^a_{t:N}}{\calR\control^a_{t:N}} \bigr]\\
			& \sbjto	&& 
				\begin{cases}
					\text{constraint } \eqref{e:augmented},\\
					\control_t\in\controlset\;\text{for all }t,\\
					\pi_{t:N} \in \text{ a class of policies}.
				\end{cases}
		\end{aligned}
		\end{equation}
\secondrevised{The objective function for stochastic systems has the same physical interpretation as for deterministic systems; for stochastic systems we focus on the average cost because of the uncertainty involved in the plant dynamics. This enables us to prevent worst-case analysis that turns out to be conservative. As in standard LQ control, the objective function has quadratic terms -- the cost-per-stage functions and the final cost function. We can choose these terms based on the penalty we want to apply on the states and the controls depending upon the physical context of the problem. Under the receding horizon implementation, the cost-per-stage $\st_i \transp Q \st_i + \control_i \transp R \control_i$ for all $i \geq 0$ shows the cost associated at each time step. Therefore, we are ultimately interested in the empirical average of the cost-per-stage, which is studied in our numerical experiments and displayed in Fig.\ 8 and Fig.\ 13. }
We shall employ a feedback from the process noise in our policies, which will lead to a modified optimal control problem. Parametrization of control policies directly in terms of the noise is standard \cite{goulart-06, ref:lofberg-03} and leads to convex problems under appropriate selection of the cost and the constraints. We impose the following blanket: 		
\begin{assumption}\label{a:blanket} 
{\rm
The communication channel between sensors and the controller is noiseless. Each component of $\wnoise_t$ is symmetrically distributed about origin.
}
\end{assumption}
Two important ingredients of \revised{stochastic model predictive control (SMPC)} are the class of control policy and the stability constraints. We shall discuss them in the following sections. The control policy class allows us to take the disturbance as feedback and to consider the effects of transmission protocols, which are discussed in \secref{s:main results}. The stability constraints allow us to transcend beyond the regime of terminal set methods. In order to ensure recursive feasibility and mean square boundedness \emph{for any positive bound on control}, the stability constraints presented in this article are different from those in \cite{ref:HokChaRamChaLyg-10}.
	\subsection{Control policy class}\label{s:problem:policy}
		We recall that the states are completely and exactly measured and acknowledgements of whether a successful control transmission has occurred or not are assumed to be causally available to the controller. The corresponding architecture is depicted in Fig. \ref{Fig:blockdia}. It is therefore possible to reconstruct the noise sequence from the sequence of observed states and control inputs with the aid of the formula
		\begin{equation}\label{e:estimation}
			\wnoise_t = \st_{t+1} - \A \st_t - \B \control^a_t,\quad t\in\Nz.
	    \end{equation}
This calculation is performed at the controller at each time \( t \in \Nz\).	    
		We follow the approach in \cite{ref:ChaHokLyg-11}, and employ \(N\)-history-dependent policies of the form
		\begin{equation}\label{e:controlt}
		\begin{aligned}
			\control_{t+\ell} &= \eta_{t+\ell} + \sum_{i=0}^{\ell-1} \theta_{\ell,t+i} \ee_{i+1}(\wnoise_{t+i}),\\ 
			                  & \teL \eta_{t+\ell} + \check{\control}_{t+\ell}, 
		\end{aligned}
		\end{equation}
		for \(\ell =0,1, \ldots, N-1\), \revised{where \(\ee_i:\R^d\lra\controlset\) is a measurable map for each \(i\) such that \(\EE[\ee_{i}(\wnoise_{t+i-1})] = 0\).}
It was shown in \cite{goulart-06} that if $\ee_i$ equals the identity map, then
there exists a (nonlinear) bijection between the above class of control
policies and the class of affine state feedback policies. Restricting control to be of the form \eqref{e:controlt} is suboptimal, but it ensures tractability of a broad
class of optimal control problems  in a computationally efficient way using convex optimization techniques \cite{cinquemani-11}.		
		 Similarly to \cite{ref:Hokayem-12, ref:HokChaRamChaLyg-10, hokayem2009stochastic, ref:ChaHokLyg-11}, the measured noise is saturated before inserting into the control vector. Hence, we can assume that there exists $\varphi_{\max} \in \R$ such that $\norm{\ee_{i}(\wnoise_{t+i-1})}_{\infty} \leq \varphi_{\max}$ for all $i=1,\cdots,N-1$. If there are no dropouts, then the control vector \(\control_{t:N}\) admits the following form under the reliable control channel:
		\begin{equation}\label{e:policy}
		\begin{aligned}
			\control_{t:N}	& = \pmat{\eta_t\\ \eta_{t+1}\\\vdots\\ \eta_{t+N-1}} + \varphi(\wnoise_{t:N-1})\Let \offset_t + \gain_t \pmat{\ee_1(\wnoise_t)\\\ee_2(\wnoise_{t+1})\\\vdots\\\ee_{N-1}(\wnoise_{t+N-2})}\\
			& \teL \offset_t + \gain_t \ee(\wnoise_{t:N-1}), 
		\end{aligned}
		\end{equation}
		where \(\offset_t\in\R^{m N}\) and \(\gain_t\) is a strictly lower block triangular matrix
		\begin{equation} \label{e:gain}
			\gain_t = \pmat{0 & 0 & \cdots & 0 & 0\\ \theta_{1, t} & 0 & \cdots & 0 & 0\\ \theta_{2, t} & \theta_{2, t+1}  & \cdots & 0 & 0\\ \vdots & \vdots & \vdots & \vdots & \vdots\\ \theta_{N-1, t} & \theta_{N-1, t+1} & \cdots & \theta_{N-1, t+N-3} & \theta_{N-1, t+N-2}},
		\end{equation}
		with each \(\theta_{k, \ell} \in \R^{m\times d}\) and $\norm{\ee(\wnoise_{t:N-1})}_{\infty} \leq \varphi_{\max}$. \revised{We shall employ saturation functions symmetric about the origin.} The matter of selection of the functions \(\ee_i\) is left open. As such we can accommodate standard saturation, piecewise linear, and sigmoidal functions, to name a few. Further discussion on $\ee_i$ as a basis element and $\theta_{\ell,t+i}$ as a Fourier coefficient can be found in \cite[\S III]{ref:ChaHokLyg-11}.
		
\section{Main results}
	\label{s:main results}		
		\revised{In this section} we propose three transmission protocols in the setting of control channel dropouts. \revised{ Each protocol leads to a different controller. We present associated constrained finite horizon optimal control problems \eqref{e:opt problem} in the form of tractable and recursively feasible optimization programs under these transmission protocols, and demonstrate that the optimization programs thus obtained are convex quadratic, and can be solved by using standard MATLAB-based software packages, e.g., YALMIP \cite{ref:lofberg-04} and the solver SDPT3-4.0 \cite{ref:toh-06}. The variance and covariance matrices used in the optimization programs are computed offline empirically to reduce the burden of online computation. Detailed examples are documented in \secref{s:simulation}. } 

\subsection{Sequential transmission of control} 
				\begin{assumptiontp}\mbox{}
		\begin{enumerate}[label={\rm (TP\arabic*)}, leftmargin=*, widest=3, align=left, start=1]
		{\rm
			\item \label{a:seq} The control values are transmitted to the actuator at every step.}
		\end{enumerate}
		\end{assumptiontp}
		For each $t=0,\reachindex,2\reachindex,\cdots$, the offset vector and gain matrix are obtained by solving the optimization problem \eqref{e:seq opt control problem} discussed below. \deleted{Akin to \eqref{e:augmented},} the control \eqref{e:controlt} is computed and transmitted at each instant. 
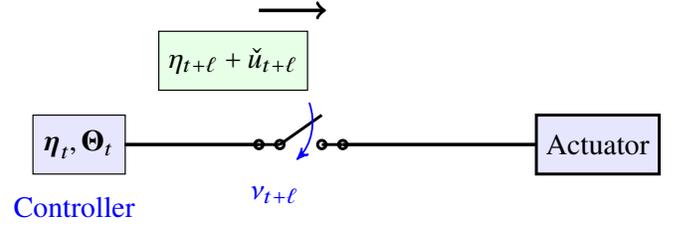
\begin{figure}[t]
\begin{adjustbox}{width = \columnwidth}
\begin{tikzpicture}
	\tikzstyle{block} = [draw, fill=blue!10, rectangle, 
	    minimum height=2em, minimum width=1.1cm]
\tikzstyle{blockgreenl} = [draw, fill=green!10, rectangle, 
	    minimum height=2em, minimum width=1cm]
	\node[coordinate] (0) at (0,0) {};
\node[anchor=north, blue] at (0.5,-0.5) {Controller};
	\node [block, right=0cm of 0] (controller) {$\offset_t,\gain_t$};
\draw[very thick ,-] (1.1,0) -- (2.7,0);
\node[coordinate] (1) at (0,1) {};
\draw[very thick ,->] (2.7,1.6) -- (3.5,1.6);
\node [blockgreenl, right=1.5cm of 1] (packet) {$\eta_{t+\ell} + \check{\control}_{t+\ell} $};
\draw[decorate, decoration=switch] (2.7cm,0cm) -- ++(1cm,0cm);
\draw[very thick ,-] (3.7,0) -- (6,0);
\node [block, right=6cm of 0] (actuator) {Actuator};
\draw[->,>=stealth',semithick, blue] (3.3,0.5) arc[radius=0.65, start angle=20, end angle=-45];
\node[anchor=east, blue] at (3.3,-0.6) {$\cnoise_{t+\ell}$};		
	\end{tikzpicture}	
	
\end{adjustbox}
\caption{Sequential transmission at time $t+\ell$}
\vspace{-1.5em}
\label{fig:policySequential}
\end{figure}				
For $\ell = 0,1, \cdots, \reachindex -1$, the control value transmitted at time $t+\ell$ is affected by dropout at the same time. Hence, $\control^a_{t:N}$ in \eqref{e:augmented} is given by   
		\begin{equation}
		\label{e:policyseq}
			\control^a_{t:N} \Let \calS\offset_t + \calS\gain_t\ee(\wnoise_{t:N-1}), \\
		\end{equation}		
		where \[ \calS \Let \pmat{I_{m}\otimes\cnoise_{t} & & \\ & \ddots & \\ & & I_{m}\otimes \cnoise_{t + \reachindex -1}\\ & & &  I_{m(N-\reachindex)} } \] and $\ee(\wnoise_{t:N-1})$ as defined in \eqref{e:policy}. \revised{In each optimization horizon of length $N$, only $\reachindex \leq N$ blocks of the stacked control vector are utilized. The structure of the matrix $\calS$, therefore, reflects no dropouts in the last $N-\reachindex$ blocks of the stacked control vector.}  
\begin{example}		
{\rm
		As an example consider $N=3,\reachindex=2$. According to this protocol, the following control sequence is \deleted{valid}:
		\begin{align*}
			\control_t &= \cnoise_t\eta_t\\
			\control_{t+1} &=\cnoise_{t+1}\eta_{t+1} + \cnoise_{t+1}\theta_{1,t}\ee_1(\wnoise_t)\\
			\control_{t+2}&=\eta_{t+2} + \theta_{2,t}\ee_1(\wnoise_t) + \theta_{2,t+1}\ee_2(\wnoise_{t+1}).
		\end{align*}		 
}
\end{example}	
		Correspondingly, we have the following optimal control problem in lieu of \eqref{e:opt problem}:
		\begin{equation}
		\label{e:seq opt control problem}
		\begin{aligned}
			& \minimize_{\offset_t, \gain_t}	&& \EE_{\st_t}\bigl[ \inprod{\st_{t:N+1}}{\calQ \st_{t:N+1}} + \inprod{\control^a_{t:N}}{\calR \control^a_{t:N}} \bigr]\\
			& \sbjto	&& \begin{cases}
				\text{constraint }\eqref{e:augmented},\\
				\text{control }\eqref{e:policyseq},\\
				\control_t\in\controlset\text{ for all }t, 
			\end{cases}
		\end{aligned}
		\end{equation}
In the above transmission protocol, if the control packet is lost at some time instant, null control is applied to the plant. The above protocol does not require any storage or computational facility at the actuator end. 
\begin{remark}
{\rm
The bit rate in transmission protocol \ref{a:seq} can be reduced by storing the received packets $\ee_{i+1}(\wnoise_{t+i})$ in a buffer near the actuator and the corresponding scaling factors $\theta_{\ell,t+i}$'s should be transmitted through the channel at each step.
}
\end{remark}
\par For notational convenience, we introduce the following definitions of covariance matrices: $\Sigma_e \Let \EE[\ee(\wnoise_{t:N-1})\ee(\wnoise_{t:N-1})\transp], \Sigma_e^{\prime} \Let \EE[\wnoise_{t:N} \ee(\wnoise_{t:N-1})\transp]$, $\Sigma_W \Let \EE[\wnoise_{t:N}\wnoise_{t:N}\transp], \mu_{\calS}=\EE[\calS]$,$ \Sigma_{\calS} = \EE[\calS\transp(\calB\transp\calQ\calB+\calR)\calS] $ and $c_t \Let \st_t\transp\calA\transp\calQ\calA\st_t + \trace(\calD\transp\calQ\calD\Sigma_W)$.
We have the following theorem:
\begin{theorem}
		\label{th:seq}
			Consider the control system \eqref{e:system} and suppose that Assumption \ref{a:blanket} holds under the transmission protocol \ref{a:seq}. Then
				 for every $t=0, \reachindex, 2\reachindex,\ldots$, the optimization problem \eqref{e:seq opt control problem} is convex, feasible, and can be rewritten as the following tractable program with tightened constraints:
\begin{align}
& \minimize_{(\offset_t, \gain_t)}\;	 \;  2\trace(\gain_t\transp\mu_{\calS}\transp\calB\transp\calQ\calD\Sigma_{\ee}^{\prime}) + 2\st_t\transp\calA\transp\calQ\calB\mu_{\calS}\offset_t \nonumber \\
 & \quad \; + \trace(\offset\transp\Sigma_{\calS}\offset_t)+
\trace((\gain_t\transp\Sigma_{\calS}\gain_t\Sigma_{\ee}) + c_t \label{e:programsingle} \\
& \sbjto:\;
 \gain_t \text{ having the structure in \eqref{e:gain}} \label{e:gainStructsingle} \\
& \quad \; \abs{\offset_t^{(i)}} + \norm{\gain_t^{(i)}}_1\varphi_{\max} \leq U_{\max} \quad \text{for all } i = 1,\ldots,Nm. \label{e:decisionboundsingle}
\end{align}
		\end{theorem}
\revised{In \cite{ref:amin-10} the authors propose burst transmission of $\reachindex$ control values every $\reachindex$ steps. In contrast, we need a modification of their protocol because the class of control policy considered in this article has feedback terms; these terms can not be transmitted in a burst since they can not be computed at the beginning of the recalculation interval.} 
	\subsection{Burst transmission of control policy parameters}
				\begin{assumptiontp}\mbox{}
		\begin{enumerate}[label={\rm (TP\arabic*)}, leftmargin=*, widest=3, align=left, start=2]
		{\rm
			\item \label{a:burst} The offset values of the control policy are transmitted in a single burst every $\reachindex$ steps to the actuator. In addition, the weighted sum of the saturated feedback terms is transmitted at each step.}
		\end{enumerate}
		\end{assumptiontp}
According to this protocol, the burst $(\offset_t)_{1:m\reachindex}$ is transmitted at time $t$ and stored in a buffer at the actuator end, where $(\offset_t)_{1:m\reachindex}$ denotes first $m\reachindex$ rows of $\offset_t$. In addition, for $\ell=1,2,\cdots,\reachindex-1$, 
$\check{\control}_{t+\ell}$ is transmitted at time $t+\ell$. The plant noise $\wnoise_{t+i}$ is calculated at the controller by \eqref{e:estimation} correctly with the help of acknowledgements. Hence, the burst $(\offset_t)_{1:m\reachindex}$ will be affected by the packet dropout occurring at time $t$, and $\check{\control}_{t+\ell}$ 
is affected by the dropouts occurring at times $t+\ell$ for $\ell=1,2,\cdots,\reachindex-1$.

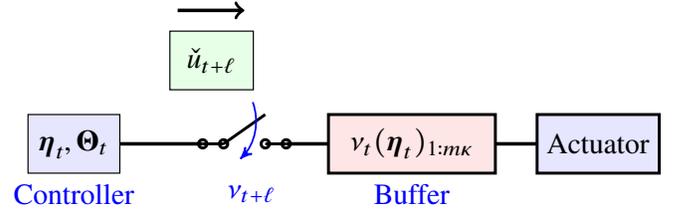
\begin{figure}[t]
\begin{adjustbox}{width = \columnwidth}
\begin{tikzpicture}
	\tikzstyle{block} = [draw, fill=blue!10, rectangle, 
	    minimum height=2em, minimum width=1.1cm]
\tikzstyle{blockgreenl} = [draw, fill=green!10, rectangle, 
	    minimum height=2em, minimum width=1cm]
	 \tikzstyle{blockgreen} = [draw, fill=green!10, rectangle, 
	    minimum height=2em, minimum width=2.7cm] 
	    \tikzstyle{blockred} = [draw, fill=red!10, rectangle, 
	    minimum height=2em, minimum width=2cm]  
	\node[coordinate] (0) at (0,0) {};
	\node [block, right=-0.1cm of 0] (controller) {$\offset_t,\gain_t$};
	\node[anchor=north, blue] at (controller.south) {Controller};
\node[coordinate] (1) at (0,1) {};
\draw[very thick ,-] (controller.east) -- (2,0);
\draw[very thick ,->] (1.7,1.6) -- (2.5,1.6);
\node [blockgreenl, right=1.6cm of 1] (packet) {$\check{\control}_{t+\ell}$};
\draw[decorate, decoration=switch] (2cm,0cm) -- ++(1cm,0cm);
\node [blockred, right=3.5cm of 0] (buffer) {$\cnoise_t(\offset_t)_{1:m\reachindex}$};
\node[anchor=north, blue] at (buffer.south) {Buffer};	
\draw[very thick ,-] (3,0) -- (buffer.west);
\draw[->,>=stealth',semithick, blue] (2.6,0.5) arc[radius=0.65, start angle=20, end angle=-45];
\node[anchor=east, blue] at (3,-0.6) {$\cnoise_{t+\ell}$};	
\node [block, right= 6 cm of 0] (actuator) {Actuator};
	\draw[very thick ,-] (buffer.east) -- (actuator.west);	
	\end{tikzpicture}	
	
\end{adjustbox}
\caption{Control channel and buffer at time $t+\ell$ for burst transmission: The red block is transmitted as a burst at time t and stored in the buffer. The green blocks are transmitted at time $t+\ell$ for $\ell=1,\cdots,\reachindex-1$ and added with corresponding stored offset value before application to the plant.}
\vspace{-1.5em}
\label{fig:policyburst}
\end{figure}
		 The control sequence at the actuator can therefore be represented in compact form as:   
        \begin{equation}\label{e:policyburst}
			\control^a_{t:N} \Let \mathcal{K}\offset_t + \calS\gain_t\ee(\wnoise_{t:N-1}), \\
		\end{equation}
		where $\calK \Let \pmat{\cnoise_t \otimes I_{m\reachindex} & 0_{m\reachindex\times m(N-\reachindex)} \\ 0_{m(N-\reachindex)\times m\reachindex} & I_{m(N-\reachindex)} }$ and $\gain_t$ and $\calS$ are given in \eqref{e:gain} and \eqref{e:policyseq},  respectively; and $\ee(\wnoise_{t:N-1})$ is as defined in \eqref{e:policy}.
\begin{example}	
{\rm	
		As an example consider $N=3,\reachindex=2$. According to this protocol, the following control sequence is valid:
		\begin{align*}
			{\control}_t &= \cnoise_t\eta_t\\
			{\control}_{t+1} &=\cnoise_t\eta_{t+1} + \cnoise_{t+1}\theta_{1,t}\ee_1(\wnoise_t)\\
			{\control}_{t+2}&=\eta_{t+2} + \theta_{2,t}\ee_1(\wnoise_t) + \theta_{2,t+1}\ee_2(\wnoise_{t+1}).
		\end{align*}
		}
\end{example}
		Correspondingly, we have the following optimal control problem in lieu of \eqref{e:opt problem}:
		\begin{equation}
		\label{e:burst opt control problem}
		\begin{aligned}
			& \minimize_{\offset_t, \gain_t}	&& \EE_{\st_t}\bigl[ \inprod{\st_{t:N+1}}{\calQ \st_{t:N+1}} + \inprod{\control^a_{t:N}}{\calR \control^a_{t:N}} \bigr]\\
			& \sbjto	&& \begin{cases}
				\text{constraint }\eqref{e:augmented},\\
				\text{control }\eqref{e:policyburst},\\
				\control_t\in\controlset\text{ for all }t. 
			\end{cases}
		\end{aligned}
		\end{equation}
The above optimization problem has decision variables $\offset_t$ and $\gain_t$ rather than a class of policy as in \eqref{e:opt problem}. The above transmission protocol needs a buffer and an adder at the actuator end. 

Let us define $\mu_{\calK} \Let \EE[\calK]$, $\Sigma_{\calK} \Let \EE[\calK\transp(\calB\transp\calQ\calB+\calR)\calK].$ Now, we have the following theorem:
\begin{theorem}\label{th:burst}
		Consider the system \eqref{e:system} and suppose that Assumption \ref{a:blanket} holds under the transmission protocol \ref{a:burst}. Then
			for every $t=0, \reachindex, 2\reachindex,\ldots$, the optimization problem \eqref{e:burst opt control problem} is convex, feasible, and can be rewritten as the following tractable program with tightened constraints:
\begin{align}\label{e:program}
			&\minimize_{\offset_t, \gain_t} \; \; \; \; \;   c_t + 2\trace(\gain_t\transp\mu_{\calS}\transp\calB\transp\calQ\calD\Sigma_{\ee}^{\prime}) + 2\st_t\transp\calA\transp\calQ\calB\mu_{\calK}\offset_t \nonumber \\ 
 & \quad \quad \quad \quad \quad + \trace(\offset_t\transp\Sigma_{\calK}\offset_t)+
\trace(\gain_t\transp\Sigma_{\calS}\gain_t\Sigma_{\ee})\\ 
			&\sbjto: \; \;  \text{Constraints } \eqref{e:gainStructsingle}, \; \eqref{e:decisionboundsingle}.\nn
\end{align}
		\end{theorem}
The transmission protocols \ref{a:seq} and \ref{a:burst} do not transmit the data repetitively. \deleted{It is quite natural to retransmit the data that may be useful in future if it has not been received at the actuator.} The next transmission protocol is motivated by the idea of repetitive transmissions \cite{Quevedo-12}.
\subsection{Sequential transmission of control along with repetitive transmission of remaining offset values}
				\begin{assumptiontp}\mbox{}
		\begin{enumerate}[label={\rm (TP\arabic*)}, leftmargin=*, widest=3, align=left, start=3]
		{\rm
			\item \label{a:repetitive} The remaining blocks of the offset vector $\offset_t$ are transmitted at each step until the first successful transmission. The control values $\control_{t+\ell}$ are transmitted at each instant.
}		
		\end{enumerate}
		\end{assumptiontp}
		 
According to this protocol, control values are transmitted at each instant, similar to the transmission protocol \ref{a:seq}. To mitigate the effects of dropouts, the components of the burst $(\offset_t)_{1:m\reachindex}$ which are useful at future instants are also transmitted repetitively until one burst is successfully received at the actuator end. A successfully received burst is stored in a buffer and the corresponding offset block is used in case of packet dropout. The buffer is emptied after each recalculation interval. For $\ell = 0,1, \cdots, \reachindex -2$, the control value along with the burst of the remaining offset components transmitted at time $t+\ell$ is affected by dropout as indicated by $\cnoise_{t+\ell}$. At time $t+\reachindex -1$, the transmitted control value is affected by the dropout $\cnoise_{t+\reachindex-1}$. The plant noise $\wnoise_{t+i}$ is calculated at the controller by \eqref{e:estimation} correctly with the help of acknowledgements.

		 The control sequence at the actuator, when \ref{a:repetitive} is adopted, can therefore be represented in compact form as:   
        \begin{equation}\label{e:policyrepetitive}
			\control^a_{t:N} \Let \mathcal{G}\offset_t + \calS\gain_t\ee(\wnoise_{t:N-1}), \\
		\end{equation}
		where \[ \calG \Let \pmat{\cnoise_t \otimes I_{m} & 0_{m\times m(N-1)} \\ 0_{m\times m} & \left( \cnoise_{t+1}+(1-\cnoise_{t+1})\cnoise_{t} \right) \otimes I_{m} & 0_{m \times m(N-2)} \\ \vdots & \vdots \\ 
0_{m(N-\reachindex)\times m\reachindex} & I_{m(N-\reachindex)}		
		} \] 
		and $\gain_t$ and $\calS$ are as given in \eqref{e:gain} and \eqref{e:policyseq},  respectively; and $\ee(\wnoise_{t:N-1})$ is as defined in \eqref{e:policy}. \revised{ Notice that the matrix $\calG$ has $(N \times (N-1))$ blocks in total, each of dimension $m \times m$. For $i = 1,\cdots,N$ and $j = 1,\cdots,N-1$, the matrix $\calG$ can be given in terms of the blocks $\calG_b(i,j)$ of each dimension $m \times m$ as follows: }
 
	\begin{align}
		\calG_b(i,j) \Let
\begin{cases}
     \rho_{t+i-1} I_m & \text{ if } i=j \leq \reachindex ,\\ 
	I_{m} & \text{ if } i =j > \reachindex ,\\
	0_m \quad  & \text{ otherwise,}			
\end{cases}		
		\end{align}
where $\rho_t = \cnoise_t$ and $\rho_{t+\ell} = \rho_{t+\ell-1} + \left( \prod_{s = 0}^{\ell -1} (1-\cnoise_{t+s}) \right) \cnoise_{t+\ell}  $. 		
\begin{figure}[t!]
\begin{adjustbox}{width = \columnwidth}
\begin{tikzpicture}
	\tikzstyle{block} = [draw, fill=blue!10, rectangle, 
	    minimum height=2em, minimum width=1.1cm]
\tikzstyle{blockgreenl} = [draw, fill=green!10, rectangle, 
	    minimum height=2em, minimum width=1cm]
	 \tikzstyle{blockgreen} = [draw, fill=green!10, rectangle, 
	    minimum height=2em, minimum width=2.7cm] 
	    \tikzstyle{blockred} = [draw, fill=red!10, rectangle, 
	    minimum height=2em, minimum width=2cm]  
	    \begin{scope}
		\node[coordinate] (0) at (0,0) {};
	\node [block, right=-0.1cm of 0] (controller) {$\offset_t,\gain_t$};
	\node[anchor=north, blue] at (controller.south) {Controller};
\draw[very thick ,-] (controller.east) -- (2,0);
\node[coordinate] (1) at (0,1) {};
\draw[very thick ,->] (1.7,2.2) -- (2.5,2.2);
\node [blockgreenl, right=1.6cm of 1] (packet) {$\eta_{t+\ell} + \check{\control}_{t+\ell}$};
\node [blockred, above=0cm of packet] (packet) {$(\offset_t)_{(\ell-1)\reachindex + 1 : m\reachindex}$};
\draw[decorate, decoration=switch] (2cm,0cm) -- ++(1cm,0cm);	
	\node [block, right=3.5cm of 0] (buffer) {empty};
\draw[very thick ,-] (3,0) -- (buffer.west);
\node[anchor=north, blue] at (buffer.south) {Buffer};
	\node [block, right=0.5cm of buffer] (actuator) {Actuator};
\draw[very thick ,-] (buffer.east) -- (actuator.west);
\draw[->,>=stealth',semithick, blue] (2.6,0.5) arc[radius=0.65, start angle=20, end angle=-45];
\node[anchor=east, blue] at (3,-0.6) {$\cnoise_{t+\ell}$};
\end{scope}
\begin{scope}[shift={(0,-3)}]
		\node[coordinate] (0) at (0,0) {};
	\node [block, right=-0.1cm of 0] (controller) {$\offset_t,\gain_t$};
	\node[anchor=north, blue] at (controller.south) {Controller};
\draw[very thick ,-] (controller.east) -- (2,0);
\node[coordinate] (1) at (0,1) {};
\draw[very thick ,->] (1.7,1.6) -- (2.5,1.6);
\node [blockgreenl, right=1.6cm of 1] (packet) {$\eta_{t+\ell} + \check{\control}_{t+\ell}$};
\draw[decorate, decoration=switch] (2cm,0cm) -- ++(1cm,0cm);	
	\node [block, right=3.5cm of 0] (buffer) {non-empty};
\draw[very thick ,-] (3,0) -- (buffer.west);
\node[anchor=north, blue] at (buffer.south) {Buffer};
	\node [block, right=0.5cm of buffer] (actuator) {Actuator};
\draw[very thick ,-] (buffer.east) -- (actuator.west);
\draw[->,>=stealth',semithick, blue] (2.6,0.5) arc[radius=0.65, start angle=20, end angle=-45];
\node[anchor=east, blue] at (3,-0.6) {$\cnoise_{t+\ell}$};
\end{scope}
	\end{tikzpicture}	
	
\end{adjustbox}
\caption{Control channel and buffer at time $t+\ell$ for \ref{a:repetitive}: The red block is transmitted as a burst only if the buffer is empty. The green blocks are transmitted at time $t+\ell$ for $\ell=0,\cdots,\reachindex-1$ and are directly applied to the plant.}
\vspace{-1em}
\label{fig:policyrep}
\end{figure}
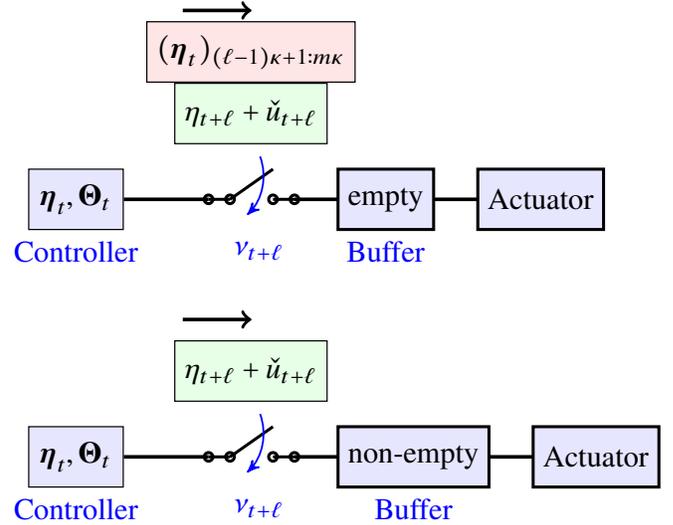
\begin{example}
{ \rm
		As an example, consider $N=3,\reachindex=2$. According to this protocol, the following control sequence is valid:
		\begin{align*}
			{\control}_t &= \cnoise_t\eta_t\\
			{\control}_{t+1} &=\cnoise_{t+1}\eta_{t+1} + (1-\cnoise_{t+1})\cnoise_{t}\eta_{t+1}+ \cnoise_{t+1}\theta_{1,t}\ee_1(\wnoise_t)\\
            {\control}_{t+2}&= \eta_{t+2} + \theta_{2,t}\ee_1(\wnoise_t) + \theta_{2,t+1}\ee_2(\wnoise_{t+1}).
		\end{align*}
}
\end{example}
		Correspondingly, we have the following optimal control problem in lieu of \eqref{e:opt problem}:
		\begin{equation}
		\label{e:rep opt control problem}
		\begin{aligned}
			& \minimize_{\offset_t, \gain_t}	&& \EE_{\st_t}\bigl[ \inprod{\st_{t:N+1}}{\calQ \st_{t:N+1}} + \inprod{\control^a_{t:N}}{\calR \control^a_{t:N}} \bigr]\\
			& \sbjto	&& \begin{cases}
				\text{constraint }\eqref{e:augmented},\\
				\text{control }\eqref{e:policyrepetitive},\\
				\control_t\in\controlset\text{ for all }t.
			\end{cases}
		\end{aligned}
		\end{equation}
\begin{remark}
{ \rm
We can observe that \eqref{e:policyrepetitive} is in the form of \eqref{e:policyburst}, hence the underlying optimization program for \ref{a:repetitive} can be easily obtained from that of \ref{a:burst} \revised{by replacing $\calK$ with $\calG$}.
}
\end{remark}
\begin{remark}
{ \rm
The bit rate in transmission protocol \ref{a:repetitive} can be reduced by storing the successfully transmitted values of saturated feedback terms in buffer. An alternative transmission protocol can be found by repetitively transmitting the future blocks of the offset parameter and past blocks of $\ee(\wnoise_{t:N-1})$ until the first successful transmission. The successfully transmitted blocks of the offset parameter and $\ee(\wnoise_{t:N-1})$ are stored in a buffer near the actuator. The corresponding blocks of the gain matrix are transmitted at each instant  along with only those blocks of $\ee(\wnoise_{t:N-1})$ which are absent in the buffer but are required to compute the control according to \eqref{e:controlt}. 
}
\end{remark}
Let us define $\mu_{\calG} \Let \EE[\calG]$, $\Sigma_{\calG} \Let \EE[\calG\transp(\calB\transp\calQ\calB+\calR)\calG].$ Now, we have the following theorem:
\begin{theorem}\label{th:rep}
		Consider the system \eqref{e:system} and suppose that Assumption \ref{a:blanket} holds under the transmission protocol \ref{a:repetitive}. Then
			for every $t=0, \reachindex, 2\reachindex,\ldots$, the optimization problem \eqref{e:burst opt control problem} is convex, feasible, and can be rewritten as the following tractable program with tightened constraints:
\begin{align}
			&\minimize_{\offset_t, \gain_t} \; \; \; \; \; c_t + 2\trace(\gain_t\transp\mu_{\calS}\transp\calB\transp\calQ\calD\Sigma_{\ee}^{\prime}) + 2\st_t\transp\calA\transp\calQ\calB\mu_{\calG}\offset_t  \nonumber \\
& \quad \quad \quad \quad \quad + \trace(\offset_t\transp\Sigma_{\calG}\offset_t)+
\trace(\gain_t\transp\Sigma_{\calS}\gain_t\Sigma_{\ee}) \\ \label{e:programrep}
			&\sbjto: \; \;  \text{Constraints } \eqref{e:gainStructsingle}, \; \eqref{e:decisionboundsingle}.\nn
\end{align}				
				
			\end{theorem}
\begin{remark}
{ \rm
Our results for mean zero stochastic noise can be easily extended to non-zero mean noise by following the approach presented in \cite{hokayem2009stochastic}.
}
\end{remark}
\begin{remark}
{ \rm
The transmitted data under three transmission protocols for the considered example with $N=3$ and $\reachindex=2$ is illustrated in Fig.\  \ref{fig:bothPolicyTime}. The block transmitted through sequential transmission \ref{a:seq} is directly applied to the plant. In contrast, in case of burst transmission \ref{a:burst}, the transmitted offset parameters are stored in a buffer. At the first instant $\eta_t$ is applied to the plant. For the second instant the saturated value of calculated disturbance $\ee_1(\wnoise_t)$ multiplied with gain parameter $\theta_{1,t}$ of the decision variable is transmitted which is added to $\eta_{t+1}$ before its application to the plant. In the case of \ref{a:repetitive}, $\eta_t$ and $\eta_{t+1}$ are transmitted at time $t$ because buffer is empty and $\eta_t$ is applied to the plant. If the transmitted data is dropped at time $t$, we need to transmit future values of offset parameter at time $t+1$  along with the control $\control_{t+1}$. In this example $\reachindex=2$, so only the control $\control_{t+1}$ is transmitted and is applied to the plant if successfully received at the actuator. 
\begin{figure}[h]
\begin{adjustbox}{width = \columnwidth}
\begin{tikzpicture}
	\tikzstyle{block} = [draw, fill=cyan!10, rectangle, 
	    minimum height=1cm, minimum width=2.63cm]
	\node[coordinate] (0) at (0,0) {};
	\node [block, right = 2.7 cm of 0] (firstb) {$\eta_t, \eta_{t+1}$};
	\node [block, right=0cm of 0] (firsts) {$\eta_t$};
	\node [block, right=5.4cm of 0] (firstr) {$\eta_t, \eta_{t+1}$};
	
	\node [block, above=0 cm of firstb] (secondb) {$\theta_{1,t}\ee_1(\wnoise_t)$};	
	
	\node [block, above=0 cm of firsts] (seconds) {$\eta_{t+1} + \theta_{1,t}\ee_1(\wnoise_t)$};
	\node[anchor=east, blue] at (0,0.1) {t};
	\node[anchor=east, blue] at (0,1.1) {t+1};	
    \node[anchor=south, blue] at (1.3,-1) {\ref{a:seq}};
	\node[anchor=south, blue] at (3.9,-1) {\ref{a:burst}};
	\node[anchor=south, blue] at (6.6,-1) {\ref{a:repetitive}};
	\node [block, above=0 cm of firstr] (secondr) {$\eta_{t+1} + \theta_{1,t}\ee_1(\wnoise_t)$};
	\end{tikzpicture}	
	
\end{adjustbox}
\caption{The transmitted blocks for $N=3$ and $\reachindex = 2$, with given optimization instant $t = 0,\reachindex,2\reachindex,\cdots$}
\vspace{-1.5em}
\label{fig:bothPolicyTime}
\end{figure}
}
\end{remark}
\revised{			
\begin{remark}
{\rm
The mean square bound of the closed-loop plant depends upon the three parameters $a,J,M$ in Theorem \ref{t:PemRos-99}, where $a$ is the expected drift in the components of the states that are bigger than $J$, and $M$ depends upon the fourth moment of the additive noise. An explicit expression for the mean square bound can be found in \cite[Corollary 2]{ref:PemRos-99} and \cite[Remark 2.7]{ref:ChaRamHokLyg-12}. In the present setting of control channel erasures, $a=\zeta p$ as obtained in proof of Lemma \ref{t:msbsingle} and Lemma \ref{t:msb}, where $\zeta \in \left]0, \frac{U_{\max}}{\sqrt{d_o}\sigma_{1}\left(\reachab_{\reachindex}(\Aortho,\Bortho)^{+}\right)} \right[$ depends upon the bound on control $U_{\max}$ and $p$ is the successful transmission probability. We demonstrate in Fig.\ \ref{Fig:msb} that the mean square bound increases with increase in the variance of the additive noise and decreases with increase in the successful transmission probability $p$.
}
\end{remark} 
}
\begin{remark}
{\rm 
Note that by setting some block elements of $\gain_t$ to zero, the size of the above optimization problems can be reduced significantly, \revised{see \cite{ref:ChaHokLyg-11}} for a discussion.
}
\end{remark}
\begin{remark}
{\rm
\secondrevised{
Transmission protocols require computation and storage facilities at the actuator end. A summary is given in Table \ref{tab:ProtocolDifference}. Notice that the transmission protocol \ref{a:seq} does not need storage or computational resources at the actuator; therefore, traditional actuators can be used for \ref{a:seq}. In case of \ref{a:burst} and \ref{a:repetitive}, we need storage facilities at the actuator, and for \ref{a:burst}, we in addition, need some computational capabilities at the actuator.}\\
\begin{table}[h]
\caption {Transmission protocols on the basis of availability of resources at actuator:} \label{tab:ProtocolDifference} 
\begin{center}
\begin{tabular}{ l | c | r }
  \hline			
   ~ & Computation at actuator & Storage at actuator \\
  \hline
  \ref{a:seq} & No & No \\
  \ref{a:burst} & Yes & Yes \\
  \ref{a:repetitive} & No & Yes \\
  \hline  
\end{tabular}
\end{center}
\end{table}
}
\end{remark}
\section{Stability} 
		\label{s:stability}	
\secondrevised{		
	It is well-known that a bounded robust positively invariant set does not exist in the context of \eqref{e:system} due to the unbounded stochastic noise \cite{ref:May-14}. Therefore, standard Lyapunov-based arguments for proving stability are not applicable. Moreover, \eqref{e:system} cannot be stabilized with the help of bounded controls if the matrix \(A\) has spectral radius greater than unity {\cite[Theorem 1.7]{ref:ChaRamHokLyg-12}}. In the presence of this fundamental limitation, we restrict our attention to Lyapunov stable plants for stability:
\begin{assumption}\label{a:stabilizability}
We assume that the matrix pair $(A,B)$ is stabilizable and the system matrix $A$ is Lyapunov stable.
\end{assumption}
Recall that a Lyapunov stable matrix has all its eigenvalues in the closed unit disk, and those on unit circle have equal algebraic and geometric multiplicities. Stability of such systems in presence of stochastic noise is not obvious.} We aim to establish the property of mean-square boundedness of the closed-loop plant. We recall the following definition:
\begin{definition}\cite[\S III.A]{chatterjee-15}
{\rm
An $\R^d$-valued random process $(\st_t)_{t \in \Nz}$ with given initial condition $\st_0 = \stinit$ is said to be mean square bounded if \[ \sup_{t \in \Nz}\EE_{\stinit}[\norm{\st_t}^2] < +\infty .\] 
}
\end{definition}  
		Let us first demonstrate the existence of a controller that ensures mean-square boundedness of the system \eqref{e:system} if \(\A\) is Lyapunov stable. Without loss of generality we assume that the pair \((\A, \B)\) is in the real Jordan canonical form. It is then a standard observation \cite{ref:HokChaRamChaLyg-10} that then the system dynamics \eqref{e:system} is of the form
\begin{equation}
		\label{e:orthogonal decomposition}
		\pmat{\stortho_{t+1} \\ \st^s_{t+1}} = \pmat{\Aortho & 0\\ 0 & \Aschur}\pmat{\stortho_{t} \\ \st^s_{t}}	+ \pmat{\Bortho\\ \Bschur}\control_t + \pmat{\wnoise_t^o \\ \wnoise_t^s}, 
		\end{equation}
		where \(\Aortho\in\R^{d_o\times d_o}\) is orthogonal and \(\Aschur\in\R^{d_s\times d_s}\) is Schur stable, with \(d = d_o + d_s\). Since the pair \((\A, \B)\) is stabilizable, there exists a positive integer \(\reachindex\) such that the pair \((\Aortho, \Bortho)\) is reachable in \(\reachindex\)-steps, i.e., \(\rank\bigl(\reachab_\reachindex(\Aortho, \Bortho)\bigr) = d_o\), where 
		\[
			\reachab_k(\Aortho, M) \Let \pmat{\Aortho^{k-1}M & \Aortho^{k-2}M & \cdots & M}
		\]
		for a matrix \(M\) of appropriate dimension. It suffices to focus on the subsystem \((\Aortho, \Bortho)\) since due to boundedness of the controls we have mean square boundedness of the Schur stable subsystem by standard arguments \cite[\S IV]{hokayem2009stochastic}.
\par Two types of drift conditions are available in general: Geometric drift conditions and constant negative drift conditions. 
\secondrevised{Geometric drift conditions are presented in \cite{meyn2012markov} and constant negative drift conditions in \cite{ref:PemRos-99}; we refer the reader to \cite[Proposition 1, Proposition 2]{chatterjee-15} for a summary and related discussion. Geometric drift conditions rely on a Lyapunov like function $f: \R^d \lra \R_{\geq 0} $ such that for $\lambda \in \; ]0,1[$ following inequality holds:
\begin{equation}\label{e:GDrift}
\EE_{\st_t} \left[ f(A\st_t + B \control_t + \wnoise_t) \right] \leq \lambda f(\st_t) \quad \text{ for all } \st_t \notin K,  
\end{equation}
for all $t \in \Nz$, where $K$ is some compact set. The construction of a function $f$ that satisfies \eqref{e:GDrift} for a given (but otherwise arbitrary) hard bound on the control inputs may not be possible. Therefore, we follow the approach in \cite{ref:ChaRamHokLyg-12} and employ following constant negative drift conditions in this article:}
for any given $r, \epsilon > 0$, $\zeta \in \left]0, \frac{U_{\max}}{\sqrt{d_o}\sigma_{1}\left(\reachab_{\reachindex}(\Aortho,\Bortho)^{+}\right)} \right[$ and for any $t=0, \reachindex, 2\reachindex, ...$, the control $u_t \in \controlset$ is chosen such that for $ j=1,2,\cdots ,d_o$ the following drift conditions hold:
\begin{align}
\EE_{\stortho_t} \left[ \left( (\Aortho^{\reachindex})\transp \reachab_{\reachindex}(\Aortho, \Bortho)\control_{t:\reachindex} \right)^{(j)} \right] & \leq -\zeta \nonumber \\
\text{ whenever }  & \left( \stortho_{t} \right)^{(j)} \geq r + \epsilon,
\label{e:drift1} \\
\EE_{\stortho_t} \left[ \left( (\Aortho^{\reachindex})\transp \reachab_{\reachindex}(\Aortho, \Bortho)\control_{t:\reachindex}\right)^{(j)} \right] & \geq \zeta \nonumber \\ 
\text{ whenever }  & \left( \stortho_{t}  \right)^{(j)}\leq -r - \epsilon. \label{e:drift2}    
\end{align}
\begin{remark}
{\rm
The above drift conditions push the $j^{th}$ component of $\left( \Aortho^{\reachindex} \right)\transp\stortho_{t+\reachindex}$ towards the origin whenever the $j^{th}$ component of $\stortho_t$ is outside the set \secondrevised{$]-r-\epsilon,r+\epsilon [$ } by application of the first $\reachindex$ elements of the control policy. 
}
\end{remark}
\begin{remark}
{\rm
The amount of constant drift $\zeta$ is chosen from the interval $\left]0, \frac{U_{\max}}{\sqrt{d_o}\sigma_{1}\left(\reachab_{\reachindex}(\Aortho,\Bortho)^{+}\right)} \right[$ \; to satisfy the hard bound on control.
}
\end{remark}
\begin{remark}
{\rm
Component-wise drift conditions are used to remove the lower bound from the control authority to extend the main results of \cite{ref:HokChaRamChaLyg-10} and \cite{ref:amin-10}.
}
\end{remark}
\revised{
\begin{remark}
{\rm
The left hand sides of the drift conditions \eqref{e:drift1} and \eqref{e:drift2} are continuous in the matrix pair $(\Aortho,\Bortho)$. Therefore, small perturbations in $(\Aortho,\Bortho)$ will lead to small variations in drift, and since $\zeta$ is strictly bigger than $0$, the sign of the drift will not change in a discontinuous fashion.    
}
\end{remark}
}
\secondrevised{
Let us consider the first $\reachindex$ blocks of \eqref{e:policy}: 
\begin{equation}\label{decision}
\control_{t: \reachindex} \Let (\offset_t)_{1: \reachindex m} +(\gain_t)_{1: \reachindex m}\ee(\wnoise_{t:N-1}).
\end{equation}
By substituting \eqref{decision} in \eqref{e:drift1} and \eqref{e:drift2}, we obtain the following stability constraints:
\begin{align}
\Bigl( (\Aortho^{\reachindex })\transp\reachab_{\reachindex}(\Aortho, \Bortho)(\offset_t)_{1: \reachindex m} \Bigr)^{(j)} & \leq -\zeta \nonumber \\ \text{whenever} & \left( \stortho_{t} \right)^{(j)} \geq r + \epsilon, \label{e:decisonConstraint1single} \\	
\Bigl( (\Aortho^{ \reachindex})\transp\reachab_{\reachindex}(\Aortho, \Bortho)(\offset_t)_{1: \reachindex m}\Bigr)^{(j)} & \geq \zeta  \nonumber \\ \text{whenever} & \left( \stortho_{t} \right)^{(j)} \leq -r - \epsilon. \label{e:decisonConstraint2single}
\end{align}
We have the following theorem: 
\begin{theorem}\label{th:stability}
Let Assumptions \ref{a:blanket} and \ref{a:stabilizability} hold, and the constraints \eqref{e:decisonConstraint1single}-\eqref{e:decisonConstraint2single} be embedded in the corresponding optimization programs of Theorems \ref{th:seq}, \ref{th:burst} and \ref{th:rep}. Then corresponding optimization problems are convex, feasible, and for any initial condition \(\stinit \in \R^{d}\) successive application of the control law given by them renders the closed-loop system mean square bounded.
\end{theorem}
}
\section{Numerical Experiments} \label{s:simulation}
		\begin{figure}[t]
\begin{adjustbox}{width=\columnwidth}
\input{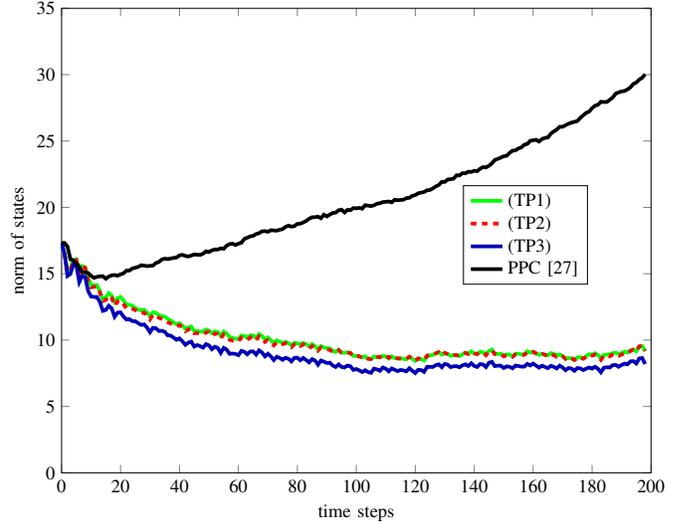}
\end{adjustbox}
\caption{Empirical average norm of states: dropout is i.i.d.}
\vspace{-1.5em}
\label{Fig:normxiid}
\end{figure}		

In this section we present simulations to illustrate our results. Consider the three dimensional linear stochastic system 
\begin{equation*}
\st_{t+1} = \pmat{0&-0.80&-0.60\\ 0.80&-0.36&0.48\\ 0.60&0.48&-0.64}\st_t + \pmat{0.16\\0.12\\0.14}\control_t + \wnoise_t, \quad
|u_t|\leq 15,
\end{equation*}
where noise sequence $\wnoise_t$ is i.i.d. Gaussian of mean zero with variance $2I_3$ and the initial condition is $\stinit=\pmat{10&10&-10} \transp$. The eigenvalues of $A$ are $\pm i$ and $-1$. So, the eigenvalues of $A$ on unit circle are semi-simple. The channel from the controller to the actuator is assumed to be of erasure type. We consider i.i.d. dropout in \secref{s:iid} and correlated dropout in \secref{s:correlated} . 
\par We solved a constrained finite-horizon optimal control problem corresponding to states and control weights \[ Q = I_3, Q_f = \pmat{12&-0.1&-0.4\\-0.1&19&-0.2\\-0.4&-0.2&2}, R =2 .\] We selected an optimization horizon, \(N=4\), recalculation interval \(N_r = \reachindex = 3 \) and simulated the system responses. We selected the nonlinear bounded term $\ee(\wnoise_{t:N-1})$ in our policy to be a vector of scalar sigmoidal functions $\varphi(\xi)=\frac{1- e^{-\xi}}{1+ e^{-\xi}}$ applied to each coordinate of the noise vector $\wnoise_t$. The covariance matrices $\Sigma_{\ee},\Sigma_W$ and $\Sigma_{\ee^{\prime}}$ that are required to solve the optimization problem were computed empirically via classical Monte Carlo methods \cite{ref:robert-13} using $10^6$ i.i.d. samples. Computations for determining our policy are carried out in the MATLAB-based software package YALMIP \cite{ref:lofberg-04} and are solved using SDPT3-4.0 \cite{ref:toh-06}. 
\par In plots for SMPC, the decision variables $\offset_t$ and $\gain_t$ are computed at times $t = 0,\reachindex,2\reachindex,\cdots$, by solving optimization problems according to Theorems \ref{th:seq}, \ref{th:burst} and \ref{th:rep}, respectively. The stability constraints \eqref{e:drift1} and \eqref{e:drift2} are satisfied by putting $r=\zeta = 0.4729$ and $\epsilon = 0.02$. The control bound $U_{\max} =15$ is satisfied by all three protocols. All the averages are taken over 300 sample paths.
\subsection{I.i.d. dropout} \label{s:iid}
The dropout is considered to be i.i.d. with successful transmission probability 0.8.
\par We compare among SMPC \ref{a:seq}, \ref{a:burst}, \ref{a:repetitive} and PPC as in \cite{Quevedo-12}. The plot of the average norm of states versus time is presented in Fig. \ref{Fig:normxiid}. The transmission protocol \ref{a:repetitive} performs better than the rest. All three protocols of SMPC perform better than PPC.
The average actuator energy and average cost are compared in Fig.\ref{Fig:ActEnergyiid} and Fig.\ref{Fig:costperstageiid}, respectively. Average cost and average actuator energy for \ref{a:seq} and \ref{a:burst} are either comparable or less than those of PPC. The transmission protocol \ref{a:repetitive} employs  more actuator energy than \ref{a:seq} and \ref{a:burst}, but less than that of PPC. However, the average cost for \ref{a:repetitive} is the lowest. \revised{Fig.\ \ref{Fig:msb} shows the variations in mean square bound with respect to successful transmission probability and variance of additive noise. The mean square bound for all protocols increases monotonically with increase in variance of process noise and decrease in successful transmission probability. Mean-square bounds are lowest for \ref{a:repetitive} in all cases.}
\begin{center}
\begin{minipage}{\columnwidth}
      \centering
      \begin{minipage}{0.46\columnwidth}
          \begin{figure}[H]
            \begin{adjustbox}{width=\columnwidth}
%
%
%
%
\begin{tikzpicture}

\begin{axis}[%
width=2.40888888888889in,
height=3.44in,
area legend,
scale only axis,
xmin=0.5,
xmax=4.5,
xtick={1,2,3,4},
xticklabels={{\ref{a:seq}},{\ref{a:burst}},{\ref{a:repetitive}},{PPC \cite{Quevedo-12}}},
ymin=0,
ymax=50
]
\addplot[ybar,bar width=0.281777777777778in,draw=black,fill=mycolor4] plot table[row sep=crcr] {%
1	33.7741744483116\\
2	33.6996753639462\\
3	36.6084504061408\\
4	49.8687222333844\\
};
\addplot [color=black,solid,forget plot]
  table[row sep=crcr]{%
0.5	0\\
4.5	0\\
};
\end{axis}
\end{tikzpicture}%
             \end{adjustbox}
              \caption{\small{Empirical average actuator energy: dropout is i.i.d.}}
              \label{Fig:ActEnergyiid}
          \end{figure}
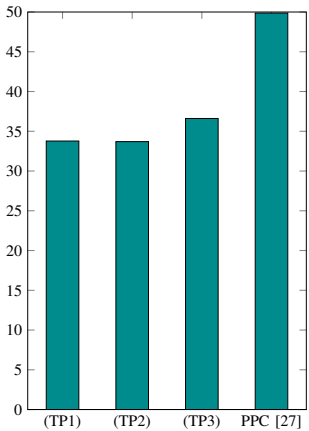
      \end{minipage}
      \quad
      \begin{minipage}{0.46\columnwidth}
          \begin{figure}[H]
          \begin{adjustbox}{width=\columnwidth}
%
%
%
%
\begin{tikzpicture}

\begin{axis}[%
width=2.40888888888889in,
height=3.47733333333333in,
area legend,
scale only axis,
xmin=0.5,
xmax=4.5,
xtick={1,2,3,4},
xticklabels={{\ref{a:seq}},{\ref{a:burst}},{\ref{a:repetitive}},{PPC \cite{Quevedo-12}}},
ymin=0,
ymax=800
]
\addplot[ybar,bar width=0.281777777777778in,draw=black,fill=mycolor4] plot table[row sep=crcr] {%
1	202.776755391265\\
2	200.671006882042\\
3	183.041948858437\\
4	745.765872167715\\
};
\addplot [color=black,solid,forget plot]
  table[row sep=crcr]{%
0.5	0\\
4.5	0\\
};
\end{axis}
\end{tikzpicture}%
              \end{adjustbox}
              \caption{\small{Empirical average cost per stage: dropout is i.i.d.}}
              \label{Fig:costperstageiid}
          \end{figure}
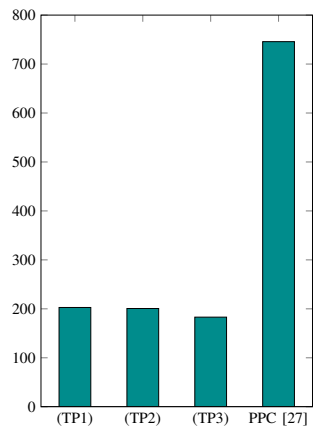
      \end{minipage}
  \end{minipage}
\end{center}  
\begin{figure}
\begin{adjustbox}{width=\columnwidth}
%
%
\begin{tikzpicture}

\begin{axis}[%
width=4.409in,
height=0.92in,
at={(0.74in,3.026in)},
scale only axis,
separate axis lines,
every outer x axis line/.append style={black},
every x tick label/.append style={font=\color{black}},
xmin=0.5,
xmax=10.3,
xtick=data,
xticklabels={0.1,0.2,0.3,0.4,0.5,0.6,0.7,0.8,0.9,1},
every outer y axis line/.append style={black},
every y tick label/.append style={font=\color{black}},
ymin=0,
ymax=10,
ylabel={\Large{log(msb)}},
axis background/.style={fill=white},
legend style={at={(0.83,0.35)},anchor=south west,legend cell align=left,align=left,draw=black}
]
 
\addplot[ybar,bar width=0.2,draw=black,fill=mycolor1,area legend] plot table[row sep=crcr] {%
0.7	4.22338934220487\\
1.7	3.93922824363018\\
2.7	3.63347532259205\\
3.7	3.45952477565365\\
4.7	3.31658743994136\\
5.7	3.08010982843064\\
6.7	2.94720695646038\\
7.7	2.86308686365951\\
8.7	2.72530250339878\\
9.7	2.5962888519847\\
};
\addlegendentry{\ref{a:seq}};
\addplot [color=black,solid,forget plot]
  table[row sep=crcr]{%
0	0\\
12	0\\
};
\addplot[ybar,bar width=0.2,draw=black,fill=mycolor2,area legend] plot table[row sep=crcr] {%
0.9	3.99480097886005\\
1.9	3.65473755637912\\
2.9	3.32907335464339\\
3.9	3.208258968273\\
4.9	3.06324020351665\\
5.9	2.91106486508849\\
6.9	2.78240143894255\\
7.9	2.70695094861573\\
8.9	2.65540269807271\\
9.9	2.54714785529048\\
};
\addlegendentry{\ref{a:burst}};
\addplot[ybar,bar width=0.2,draw=black,fill=mycolor3,area legend] plot table[row sep=crcr] {%
1.1	3.60242738099599\\
2.1	3.25775812773477\\
3.1	3.05109462282612\\
4.1	2.97013987580828\\
5.1	2.80747084643315\\
6.1	2.77985857612557\\
7.1	2.68701657445758\\
8.1	2.65898177566892\\
9.1	2.62618737790789\\
10.1	 2.54714785600379\\
};
\addlegendentry{\ref{a:repetitive}};
\node[draw] at (78,85) {$\Sigma_{\wnoise} = 0.1I_3$};

\end{axis}

\begin{axis}[%
width=4.409in,
height=0.92in,
at={(0.74in,1.748in)},
scale only axis,
separate axis lines,
every outer x axis line/.append style={black},
every x tick label/.append style={font=\color{black}},
xmin=0.5,
xmax=10.3,
xtick=data,
xticklabels={0.1,0.2,0.3,0.4,0.5,0.6,0.7,0.8,0.9,1},
every outer y axis line/.append style={black},
every y tick label/.append style={font=\color{black}},
ymin=0,
ymax=10,
ylabel={\Large{log(msb)}},
axis background/.style={fill=white}
]

\addplot[ybar,bar width=0.2,draw=black,fill=mycolor1,area legend] plot table[row sep=crcr] {%
0.7	6.80748372645801\\
1.7	6.50924281933943\\
2.7	6.20768738390114\\
3.7	6.00704747996702\\
4.7	5.74706281836283\\
5.7	5.56161961255932\\
6.7	5.38845358386494\\
7.7	5.19321597012598\\
8.7	5.06741762220959\\
9.7	4.88334773783803\\
};
\addplot [color=black,solid,forget plot]
  table[row sep=crcr]{%
0	0\\
12	0\\
};
\addplot[ybar,bar width=0.2,draw=black,fill=mycolor2,area legend] plot table[row sep=crcr] {%
0.9	6.60264881188527\\
1.9	6.26629991630807\\
2.9	6.0068573398257\\
3.9	5.80567436595016\\
4.9	5.57859069346613\\
5.9	5.38042937668845\\
6.9	5.21321521419547\\
7.9	5.08047299973552\\
8.9	4.92409988423503\\
9.9	4.76864623091547\\
};
\addplot[ybar,bar width=0.2,draw=black,fill=mycolor3,area legend] plot table[row sep=crcr] {%
1.1	6.24647875152364\\
2.1	5.79001330860528\\
3.1	5.55763012829857\\
4.1	5.35174183038032\\
5.1	5.21479512333938\\
6.1	5.09003084180082\\
7.1	4.98875476869025\\
8.1	4.94039550528453\\
9.1	4.85106619519136\\
10.1	4.76864623183385\\
};
\node[draw] at (80,89) {$\Sigma_{\wnoise} = I_3$};
\end{axis}

\begin{axis}[%
width=4.409in,
height=0.92in,
at={(0.74in,0.469in)},
scale only axis,
separate axis lines,
every outer x axis line/.append style={black},
every x tick label/.append style={font=\color{black}},
xmin=0.5,
xmax=10.3,
xtick=data,
xticklabels={0.1,0.2,0.3,0.4,0.5,0.6,0.7,0.8,0.9,1},
every outer y axis line/.append style={black},
every y tick label/.append style={font=\color{black}},
ymin=0,
ymax=11,
xlabel={\Large{p}},
ylabel={\Large{log(msb)}},
axis background/.style={fill=white}
]
\addplot[ybar,bar width=0.2,draw=black,fill=mycolor1,area legend] plot table[row sep=crcr] {%
0.7	9.25004443363967\\
1.7	8.97018457027959\\
2.7	8.7291233343371\\
3.7	8.56370330676915\\
4.7	8.40833120200576\\
5.7	8.27785893363954\\
6.7	8.11907524217513\\
7.7	7.97964382491735\\
8.7	7.89442045804099\\
9.7	7.75438270616841\\
};

\addplot [color=black,solid,forget plot]
  table[row sep=crcr]{%
0	0\\
12	0\\
};
\addplot[ybar,bar width=0.2,draw=black,fill=mycolor2,area legend] plot table[row sep=crcr] {%
0.9	9.25746073062812\\
1.9	9.00131647939932\\
2.9	8.75864420009489\\
3.9	8.5739607331837\\
4.9	8.40643185866636\\
5.9	8.22500596591792\\
6.9	8.12411938806138\\
7.9	7.996333811941\\
8.9	7.88897613387395\\
9.9	7.75438270616841\\
};

\addplot[ybar,bar width=0.2,draw=black,fill=mycolor3,area legend] plot table[row sep=crcr] {%
1.1	8.9574421420464\\
2.1	8.62885336727047\\
3.1	8.37356876371893\\
4.1	8.23540273030149\\
5.1	8.10131016994789\\
6.1	7.98352369139941\\
7.1	7.90304828124699\\
8.1	7.83769656740883\\
9.1	7.80202939527337\\
10.1	7.7543825932316\\
};
\node[draw] at (80,96) {$\Sigma_{\wnoise} = 10I_3$};
\end{axis}

\end{tikzpicture}%
\end{adjustbox}
\caption{Empirical mean square bound (msb) in log scale as a function of the successful transmission probability $p$ and the variance of the additive noise.}
\label{Fig:msb}
\end{figure}
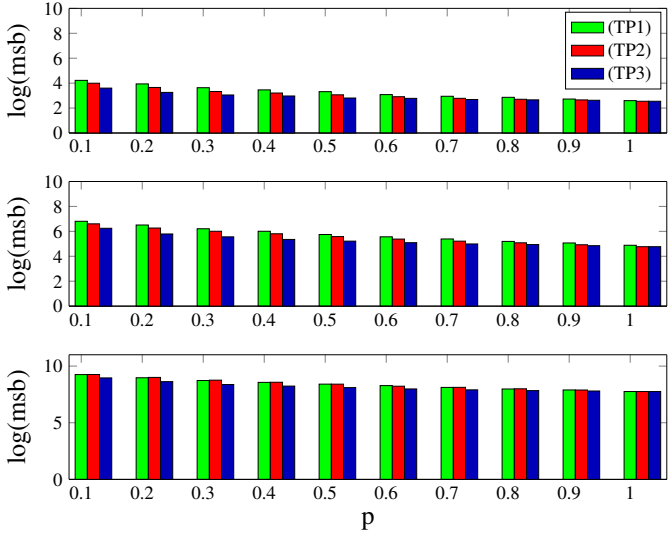	
\subsection{Correlated dropout}\label{s:correlated}
Following \cite{queahl13a}, we model the \emph{communication reliability} 
via an underlying \emph{network state} process $(\Xi_t)_{t\in\Nz}$ taking on values
in a finite set $ \bbB=\{1,2, \dots,|\bbB|\},$ where $|\bbB|$ is the cardinality of $\bbB$. Each $i\in \bbB$ corresponds to a different environmental
condition (such as network congestion or positions of mobile objects). For a given network state  the dropouts are modelled
as i.i.d.\ with probabilities
\begin{equation}
  \label{eq:2}
  p_i=\PP\{\cnoise_t =1 \,|\, \Xi_t = i\},\quad i\in \bbB.
\end{equation}
Correlations can be captured in this model by allowing $(\Xi_t)_{t\in\Nz}$ to be
time-homogeneous Markovian
with transition probabilities:
\begin{equation}
  \label{eq:3}
   p_{ij} = \PP\big\{\Xi_{t+1} = j \, \big| \, \Xi_t = i\big\},\quad \text{ for all } i,j \in
  \bbB \text{ and } \;t\in\Nz.
\end{equation}
\begin{figure}
\centering
\begin{adjustbox}{width=0.8\columnwidth}
\input{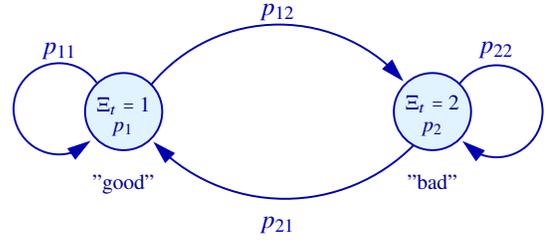}
\end{adjustbox}
\caption{Transmission dropout model with a binary network state
    $(\Xi_t)_{t\in\Nz}$: when $\Xi_t=1$ the channel is reliable with very low
  dropout probabilities; $\Xi_t=2$ refers to a situation where the channel is
  unreliable and transmissions are more likely to be dropped.}
 \vspace{-1em} 
\label{fig:markov_nmpc}
\end{figure} 
Fig.\ref{fig:markov_nmpc} illustrates a simplified version of the model adopted in \cite{queahl13a}.
The \emph{network state} process encompasses Markovian dropouts and
i.i.d.\ dropouts as special cases.  

We considered the packet dropout model in Fig.\
\ref{fig:markov_nmpc}: When the channel is in ``good'' state successful transmission probability $p_1=0.8$, whereas in the ``bad'' state that is $p_2=0.4$. The transition probability to switch from a bad channel state to the good
state is taken as $p_{21}=0.9$; the failure rate is taken as $p_{12}=0.3$. 
\par The plot of average norm of states is presented in Fig. \ref{Fig:normx}. The transmission protocol \ref{a:repetitive} performs better than the rest. All three protocols of SMPC perform better than PPC.
\begin{figure}
\begin{adjustbox}{width=\columnwidth}
\input{normx.tex}
\end{adjustbox}
\caption{Empirical average norm of states: correlated dropout}
\vspace{-1.5em}
\label{Fig:normx}
\end{figure}
The average actuator energy and average cost are compared in Fig.\ref{Fig:ActEnergy} and Fig.\ref{Fig:costperstage}, respectively. Average cost and average actuator energy for \ref{a:seq} and \ref{a:burst} are comparable and less than those of PPC. Transmission protocol \ref{a:repetitive} employs  more actuator energy than \ref{a:seq}, \ref{a:burst} but less than PPC. However, the average cost for \ref{a:repetitive} is the lowest. 

\begin{center}
\begin{minipage}{\columnwidth}
      \centering
      \begin{minipage}{0.46\columnwidth}
          \begin{figure}[H]
            \begin{adjustbox}{width=\columnwidth}
%
%
%
%
\begin{tikzpicture}[yscale = 0.98]

\begin{axis}[%
width=2.40888888888889in,
height=3.47in,
area legend,
scale only axis,
xmin=0.5,
xmax=4.5,
xtick={1,2,3,4},
xticklabels={{\ref{a:seq}},{\ref{a:burst}},{\ref{a:repetitive}},{PPC \cite{Quevedo-12}}},
ymin=0,
ymax=50
]
\addplot[ybar,bar width=0.281777777777778in,draw=black,fill=mycolor4] plot table[row sep=crcr] {%
1	32.8595691010748\\
2	32.2148194801483\\
3	37.3403567359687\\
4	48.5055041941542\\
};
\addplot [color=black,solid,forget plot]
  table[row sep=crcr]{%
0.5	0\\
4.5	0\\
};
\end{axis}
\end{tikzpicture}%
             \end{adjustbox}
              \caption{\small{Empirical average actuator energy: correlated dropout}}
              \label{Fig:ActEnergy}
          \end{figure}
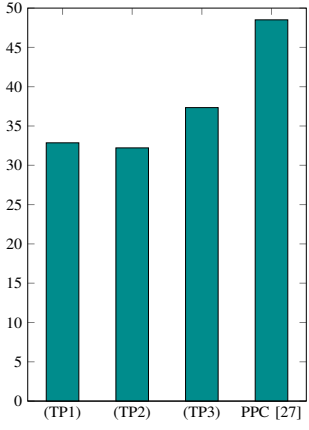
      \end{minipage}
      \quad
      \begin{minipage}{0.46\columnwidth}
          \begin{figure}[H]
          \begin{adjustbox}{width=\columnwidth}
%
%
%
%
\begin{tikzpicture}

\begin{axis}[%
width=2.40888888888889in,
height=3.47733333333333in,
area legend,
scale only axis,
xmin=0.5,
xmax=4.5,
xtick={1,2,3,4},
xticklabels={{\ref{a:seq}},{\ref{a:burst}},{\ref{a:repetitive}},{PPC \cite{Quevedo-12}}},
ymin=0,
ymax=600
]
\addplot[ybar,bar width=0.281777777777778in,draw=black,fill=mycolor4] plot table[row sep=crcr] {%
1	240.768612557164\\
2	234.858521982851\\
3	200.917768868476\\
4	567.959240565643\\
};
\addplot [color=black,solid,forget plot]
  table[row sep=crcr]{%
0.5	0\\
4.5	0\\
};
\end{axis}
\end{tikzpicture}%
              \end{adjustbox}
              \caption{\small{Empirical average cost per stage: correlated dropout}}
              \label{Fig:costperstage}
          \end{figure}
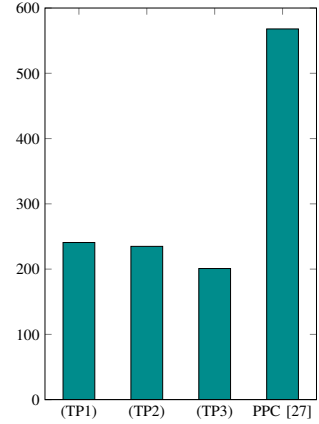
      \end{minipage}
  \end{minipage}
\end{center}  
\section{Conclusion}\label{s:conclusion}
This article presents convex and feasible optimal control formulations for receding horizon control of Lyapunov stable plants in the presence of stochastic noise and control channel erasures. Three transmission protocols are suggested which have different requirements on communication use and computation capability at the actuator. The construction of the optimization programs is systematic.
For  all protocols presented in this article, mean square boundedness of the closed-loop states and satisfaction of hard bound on control is guaranteed and verified by simulation results. In particular, the methods proposed outperform deterministic controllers available in the literature.

\section{Appendix}
\label{s:appendix}
	This appendix provides the proofs of Theorem \ref{th:seq}, \ref{th:burst}, \ref{th:rep} and \ref{th:stability}.
Let us  define the component-wise saturation function  \(\R^{d_o}\ni z\longmapsto \sat_{r, \zeta}^\infty(z)\in\R^{d_o}\) to be 
			\[
				\bigl(\sat_{r, \zeta}^\infty(z)\bigr)^{(i)} = \begin{cases}
					z_i \zeta/r	& \text{if \(\abs{z_i} \le r\),}\\
					\zeta	 	& \text{if \(z_i > r\), and }\\
					-\zeta		& \text{otherwise,}
				\end{cases}
			\]
			for each \(i = 1, \ldots, d_o\).
			\par We need the following basic result:
\begin{theorem}[{\cite[Theorem 2.1]{ref:PemRos-99}}]
		\label{t:PemRos-99}
			Let \((X_t)_{t\in\Nz}\) be a family of real valued random variables on a probability space \((\Omega, \sigalg, \PP)\), adapted to a filtration \((\sigalg_t)_{t\in\Nz}\). Suppose that there exist scalars \(J, M, a > 0\) such that \(X_0 < J\), and 
			\[
				\EE^{\sigalg_t}[X_{t+1} - X_t] \le -a \quad\text{whenever } X_t > J,
			\]
			and
			\[
				\EE\bigl[\abs{X_{t+1} - X_t}^4\,\big|\, X_0, \ldots, X_t\bigr] \le M\quad\text{for all }t\in\Nz.
			\]
			Then there exists a constant \(C > 0\) such that \(\sup_{t\in\Nz}\EE\bigl[ \left((X_t)_+\right)^2 \bigr] \le C\).
		\end{theorem}
\begin{lemma}\label{t:msbsingle}
Consider the orthogonal part of the system \eqref{e:system} given by
\begin{equation}
			\label{e:ortho system}
				\stortho_{t+1} = \Aortho \stortho_t + \Bortho \control^a_t + \wnoise_t^o
			\end{equation}			
as per the decomposition \eqref{e:orthogonal decomposition} above. Suppose the pair $(\Aortho,\Bortho)$ is $\reachindex-$step reachable. Let $U_{\max} > 0$ be given such that $\norm{\control_t}_{\infty} \leq U_{\max}$ for all $t$. Then there exists a \(\reachindex\)-history dependent policy under the transmission protocol  \ref{a:seq} that renders the system \eqref{e:ortho system} mean-square bounded.
		\end{lemma}
		\begin{proof}
Consider the \(\reachindex\)-subsampled system \eqref{e:ortho system}:
			\begin{align}
				\stortho_{\reachindex (t+1)}	& = \Aortho^\reachindex \stortho_{\reachindex t} + \reachab_{\reachindex}(\Aortho, \Bortho) \pmat{\cnoise_{\reachindex t}\control_{\reachindex t}\\ \vdots \\ \cnoise_{\reachindex(t+1)-1}\control_{\reachindex(t+1) - 1}}\nn\\
					& \quad + \reachab_{\reachindex}(\Aortho, I_{d_o}) \pmat{\wnoise^o_{\reachindex t}\\ \vdots \\ \wnoise^o_{\reachindex (t+1) - 1}}\nn\\
					& \teL \Aortho^\reachindex \stortho_{\reachindex t} + \reachab_\reachindex(\Aortho, \Bortho) \control^a_{\reachindex t:\reachindex}\\
					& \qquad + \reachab_\reachindex(\Aortho, I_{d_o}) \wnoise^o_{\reachindex t:\reachindex}.\nn
			\label{e:subsampled system}
			\end{align}
			Define \(y_{\reachindex t} \Let (\Aortho\transp)^{\reachindex t} \stortho_{\reachindex t}\) for each \(t\in\Nz\). It follows that
\begin{equation}\label{eq:tildeukt}			
			\begin{aligned} 
				y_{\reachindex(t+1)} & = y_{\reachindex t} + (\Aortho\transp)^{\reachindex(t+1)} \reachab_\reachindex(\Aortho, \Bortho) \control^a_{\reachindex t:\reachindex}\\ 
					& \qquad + (\Aortho\transp)^{\reachindex(t+1)} \reachab_\reachindex(\Aortho, I_{d_o}) \wnoise^o_{\reachindex t:\reachindex}\\ 
&=y_{\reachindex t} + \tilde{\control}_{\reachindex t} + (\Aortho\transp)^{\reachindex(t+1)} \reachab_\reachindex(\Aortho, I_{d_o}) \wnoise^o_{\reachindex t:\reachindex},
			\end{aligned}
			\end{equation}
			where, $ \tilde{\control}_{\reachindex t} = (\Aortho\transp)^{\reachindex(t+1)} \reachab_\reachindex(\Aortho, \Bortho) \pmat{\cnoise_{\reachindex t}\control_{\reachindex t}\\ \cnoise_{\reachindex t+1}\control_{\reachindex t+1}\\ \vdots \\ \cnoise_{\reachindex t+\reachindex-1}\control_{\reachindex t+\reachindex-1} }. $
			Let \(\sigalg_{\reachindex t}\) denote the \(\sigma\)-algebra generated by \(\{\st_{\reachindex \ell}\,|\, \ell = 0, 1, \ldots, t\}\). From \eqref{eq:tildeukt}, we have
	\[ \EE^{\sigalg_{\reachindex t}}\bigl[ y_{\reachindex (t+1)} - y_{\reachindex t}\bigr] = \EE^{\sigalg_{\reachindex t}} [\tilde{\control}_{\reachindex t} ]. \]
			We claim that the sequence
			\begin{align}\label{e:suggested}				
				\control_{\reachindex t:\reachindex} = - \reachab_{\reachindex}(\Aortho, \Bortho)^+ \Aortho^{\reachindex (t+1)} \sat_{r, \zeta}^\infty \bigl((\Aortho\transp)^{\reachindex t} \stortho_{\reachindex t}\bigr).
			\end{align}
is feasible for the problem \eqref{e:programsingle}.
Let us verify the hypotheses of Theorem \ref{t:PemRos-99}. 			For the \(i\)-th component \(y^{(i)}_{\reachindex t}\) of \(y_{\reachindex t}\) we see that 
			\[
				\EE^{\sigalg_{\reachindex t}}\bigl[ y^{(i)}_{\reachindex (t+1)} - y^{(i)}_{\reachindex t}\bigr] = -p\zeta \quad\text{whenever }y^{(i)}_{\reachindex t} > r,
			\]
			and similarly
			\[
				\EE^{\sigalg_{\reachindex t}}\bigl[ y^{(i)}_{\reachindex (t+1)} - y^{(i)}_{\reachindex t} \bigr] = p\zeta \quad\text{whenever }y^{(i)}_{\reachindex t} < -r.
			\]
			A straightforward computation relying on uniform boundedness of the control shows that there exists \(M > 0\) such that
			\[
				\EE\Bigl[\abs{y^{(i)}_{\reachindex (t+1)} - y^{(i)}_{\reachindex t}}^4\,\Big|\, y^{(i)}_{\reachindex t}, \ldots, y^{(i)}_0\Bigr] \le M\quad\text{for all }t.
			\]
			Theorem \ref{t:PemRos-99} now guarantees the existence of constants \(C^{(i)}_+, C^{(i)}_- > 0\), \(i = 1, \ldots, d\), such that
			\[
				\sup_{t\in\Nz} \EE_{\stinit}\Bigl[ \left( \bigl(y^{(i)}_{\reachindex t}\bigr)_+ \right)^2\Bigr] \le C^{(i)}_+\quad\text{and}\quad \sup_{t\in\Nz} \EE_{\stinit}\Bigl[ \left( \bigl(y^{(i)}_{\reachindex t}\bigr)_-\right)^2\Bigr] \le C^{(i)}_-.
			\]
			Since \(\abs{y} = y_+ + y_- = y_+ + (-y)_+\) for any \(y\in\R\), and also \(\norm{y}^2 = \sum_{i=1}^{d_o} \abs{y^{(i)}}^2 \le 2 \sum_{i=1}^{d_o} \bigl((y^{(i)}_+)^2 + (y^{(i)}_-)^2\bigr)\), we see at once that the preceding bounds imply
			\[
				\sup_{t\in\Nz} \EE_{\stinit}\bigl[\norm{y_{\reachindex t}}^2\bigr] < C\quad\text{for some constant \(C > 0\)}.
			\]
			Since \(\stortho_t\) is derived from \(y_t\) by an orthogonal transformation, it immediately follows that
			\[
				\sup_{t\in\Nz} \EE_{\stinit}\bigl[\norm{\stortho_{\reachindex t}}^2\bigr] \le C.
			\]
			A standard argument (e.g., as in \cite{ref:RamChaMilHokLyg-10}) now suffices to conclude from mean-square boundedness of the \(\reachindex\)-subsampled process \((\stortho_{\reachindex t})_{t\in\Nz}\) the same property of the original process \((\stortho_{t})_{t\in\Nz}\).
		\end{proof}
\begin{lemma}\label{t:objectivesingle}
Consider the system \eqref{e:system}, for every $t=0, \reachindex, 2\reachindex, \cdots,$ the problem \eqref{e:opt problem} under policy \eqref{e:policy}, transmission protocol \ref{a:seq} and control set \eqref{e:controlset} is convex quadratic with respect to the decision variable $(\offset_t,\gain_t)$. The objective function \eqref{e:seq opt control problem} is given by \eqref{e:programsingle}.
\end{lemma}	
\begin{proof}
It is clear from the construction that $\calQ$ is symmetric and non-negative definite and $\calR$ is symmetric and positive definite. Hence, $\st_{t:N+1}\transp\calQ\st_{t:N+1}+\control^a_{t:N}\calR \control^a_{t:N}$ is convex quadratic. Both $\st_{t:N+1}$ and $\control^a_{t:N}$ are affine function of the design parameters $(\offset_t,\gain_t)$ for any realization of noise $\wnoise_{t:N}$. It directly follows that $V_t$ is convex quadratic in $(\offset_t,\gain_t)$. The control constraint set in \eqref{e:controlset} is convex affine in $(\offset_t,\gain_t)$. Hence the given problem \eqref{e:opt problem} is a convex quadratic program. 
The objective function of \eqref{e:seq opt control problem} is given by
\begin{align*}
&\EE_{\st_{t}}\bigl[ \inprod{\st_{t:N+1}}{\calQ \st_{t:N+1}} + \inprod{\control^a_{t:N}}{\calR\control^a_{t:N}}\bigr]\\
&=\EE_{\st_{t}}\Bigl[ 
\Bigl(\calA \st_t + \calB \control^a_{t:N} + \calD \wnoise_{t:N}\Bigr) \transp \calQ \Bigl( \calA \st_t  + \calB \control^a_{t:N} \\ & \quad + \calD \wnoise_{t:N}\Bigr) 
+ (\control^a_{t:N})\transp \calR\control^a_{t:N} \Bigr]\\
&=\EE_{\st_{t}}\Bigl[
\st_t\transp\calA\transp\calQ\calA\st_t + (\control^a_{t:N})\transp(\calB\transp\calQ\calB+\calR)\control^a_{t:N} \\ & \quad + \wnoise_{t:N}\transp\calD\transp\calQ\calD\wnoise_{t:N} + 2\bigl( \st_t\transp\calA\transp\calQ\calB \control^a_{t:N} + \st_t\transp\calA\transp\calQ\calD \wnoise_{t:N} \\ & \quad + (\control^a_{t:N})\transp\calB\transp\calQ\calD \wnoise_{t:N}    \bigr) \Bigr]\\
&=\EE_{\st_{t}}\bigl[ \st_t\transp\calA\transp\calQ\calA\st_t + \wnoise_{t:N}\transp\calD\transp\calQ\calD\wnoise_{t:N} + 2\st_t\transp\calA\transp\calQ\calD\wnoise_{t:N}\bigr] \\ & \quad +2\EE_{\st_{t}}\bigl[ (\calS\offset_t + \calS\gain_t\ee(\wnoise_{t:N-1}))\transp\calB\transp\calQ\calD\wnoise_{t:N}
\bigr]\\ & \quad +\EE_{\st_{t}}\bigl[ 2\st_t\transp\calA\transp\calQ\calB(\calS\offset_t + \calS\gain_t\ee(\wnoise_{t:N-1}))
\bigr] +\EE_{\st_{t}}\bigl[ (\calS\offset_t \\ & \quad +\calS\gain_t\ee(\wnoise_{t:N-1}))\transp(\calB\transp\calQ\calB+\calR)(\calS\offset_t + \calS\gain_t\ee(\wnoise_{t:N-1}))
\bigr]\\
&=\st_t\transp\calA\transp\calQ\calA\st_t+\trace(\calD\transp\calQ\calD\Sigma_W) 
+2(0) + 0 \\ & \quad + 2\EE_{\st_{t}}\bigl[ (\calS\gain_t\ee(\wnoise_{t:N-1}))\transp\calB\transp\calQ\calD\wnoise_{t:N} \bigr] \\ & \quad +2\st_t\transp\calA\transp\calQ\calB\EE_{\st_{t}}\bigl[ \calS\bigr] \offset_t +0 + \EE_{\st_{t}}\bigl[ (\calS\offset_t)\transp(\calB\transp\calQ\calB+\calR)\calS\offset_t \\ & \quad + (\calS\gain_t\ee(\wnoise_{t:N-1}))\transp(\calB\transp\calQ\calB+\calR)(\calS\gain_t\ee(\wnoise_{t:N-1})) 
\bigr]+0\\
&= 2\trace(\gain_t\transp\mu_{\calS}\calB\transp\calQ\calD\Sigma_{\ee}^{\prime}) + 2\st_t\transp\calA\transp\calQ\calB\mu_{\calS}\offset_t  + \trace(\offset\transp\EE[\calS\transp(\calB\transp\calQ\calB \\ & \quad +\calR)\calS]\offset_t)+
\trace((\gain_t\transp\EE[\calS\transp(\calB\transp\calQ\calB+\calR)\calS]\gain_t\Sigma_{\ee}) + c_t
\\ 
&= c_t + 2\trace(\gain_t\transp\mu_{\calS}\transp\calB\transp\calQ\calD\Sigma_{\ee}^{\prime}) + 2\st_t\transp\calA\transp\calQ\calB\mu_{\calS}\offset_t  + \trace(\offset\transp\Sigma_{\calS}\offset_t) \\ & \quad +
\trace((\gain_t\transp\Sigma_{\calS}\gain_t\Sigma_{\ee})
\end{align*}		
\end{proof}			
\begin{lemma}\label{t:msb}
Consider the orthogonal part of the system \eqref{e:system} as per the decomposition \eqref{e:orthogonal decomposition} above given by \eqref{e:ortho system}. Let \(\reachindex\) denote the reachability index of the pair \((\Aortho,\Bortho)\). Let $U_{\max} > 0$ be given such that $\norm{\control_t}_{\infty} \leq U_{\max}$ for all $t$. Then there exists a \(\reachindex\)-history dependent policy under the transmission protocol \ref{a:burst} that renders the system \eqref{e:ortho system} mean-square bounded.
		\end{lemma}
\begin{proof}
Consider the \(\reachindex\)-subsampled system \eqref{e:ortho system}:
			\begin{align}\label{e:subsampled system burst}
				\stortho_{\reachindex (t+1)}	& = \Aortho^\reachindex \stortho_{\reachindex t} + \reachab_{\reachindex}(\Aortho, \Bortho) \pmat{\control_{\reachindex t}\\ \vdots \\ \control_{\reachindex(t+1) - 1}}\cnoise_{\reachindex t}\nn\\
					& \quad + \reachab_{\reachindex}(\Aortho, I_{d_o}) \pmat{\wnoise^o_{\reachindex t}\\ \vdots \\ \wnoise^o_{\reachindex (t+1) - 1}}
			\end{align}
			Define \(y_{\reachindex t} \Let (\Aortho\transp)^{\reachindex t} \stortho_{\reachindex t}\) for each \(t\in\Nz\). It follows that
			\begin{align*}
				y_{\reachindex(t+1)} & = y_{\reachindex t} + (\Aortho\transp)^{\reachindex(t+1)} \reachab_\reachindex(\Aortho, \Bortho) \control_{\reachindex t:\reachindex}\cnoise_{\reachindex t}\\
					& \qquad + (\Aortho\transp)^{\reachindex(t+1)} \reachab_\reachindex(\Aortho, I_{d_o}) \wnoise^o_{\reachindex t:\reachindex}.
			\end{align*}
			We substitute \eqref{e:suggested} in above equation.
Then
			\begin{align*}
				y_{\reachindex(t+1)} & = y_{\reachindex t} + (\Aortho\transp)^{k(t+1)} \reachab_{\reachindex}(\Aortho, \Bortho) \times \\
				& \left(-\reachab_\reachindex(\Aortho, \Bortho)^+ \Aortho^{\reachindex (t+1)} \sat_{r, \zeta}^\infty\bigl((\Aortho\transp)^{\reachindex t} \stortho_{\reachindex t}\bigr)\right)\cnoise_{\reachindex t} + \\
				& \qquad + (\Aortho\transp)^{\reachindex(t+1)} \reachab_\reachindex(\Aortho, I_{d_o}) \wnoise^o_{\reachindex t:\reachindex}.
			\end{align*}

Let us verify the hypotheses of Theorem \ref{t:PemRos-99}. Let \(\sigalg_{\reachindex t}\) denote the \(\sigma\)-algebra generated by \(\{\st_{\reachindex \ell}\,|\, \ell = 0, 1, \ldots, t\}\). 
			For the \(i\)-th component \(y^{(i)}_{\reachindex t}\) of \(y_{\reachindex t}\) we see that 
			\begin{align*}
				& \EE^{\sigalg_{\reachindex t}}\bigl[ y^{(i)}_{\reachindex (t+1)} - y^{(i)}_{\reachindex t}\bigr]\\
				& = \EE^{\sigalg_{\reachindex t}}\Bigl[\bigl((\Aortho\transp)^{\reachindex(t+1)} \reachab_\reachindex(\Aortho, \Bortho) (\control)_{\reachindex t:\reachindex}\cnoise_{\reachindex t}\\
				& \qquad + (\Aortho\transp)^{\reachindex(t+1)} \reachab_\reachindex(\Aortho, I_{d_o}) \wnoise^o_{\reachindex t:\reachindex}\bigr)^{(i)}\Bigr]\\
				& = p\EE^{\sigalg_{\reachindex t}}\Bigl[ \sat_{r, \zeta}^{\infty}\bigl(y^{(i)}_{\reachindex t}\bigr)\Bigr],
			\end{align*}
where $p = \EE^{\sigalg_{\reachindex t}}[\cnoise_{\reachindex t}]$. 
Rest of the proof follows similarly as in Lemma \ref{t:msbsingle}.
		\end{proof}
\begin{lemma}\label{t:objective}
Consider the system \eqref{e:system}, for every $t=0, \reachindex, 2\reachindex, \cdots,$ the problem \eqref{e:opt problem} under policy \eqref{e:policy}, transmission protocol \ref{a:burst} and control set \eqref{e:controlset} is convex quadratic with respect to the decision variable $(\offset_t,\gain_t)$. The objective function \eqref{e:burst opt control problem} is given by \eqref{e:program}.
\end{lemma}
\begin{proof}
\secondrevised{
By the same argument given in the proof of Lemma \ref{t:objectivesingle}, it follows that $V_t$ is convex quadratic in $(\offset_t,\gain_t)$. The objective function of \eqref{e:burst opt control problem} is also derived similarly.}
\end{proof}			


\begin{proof}[Proof of Theorem \ref{th:seq}]
The proof of convexity and the formulation of the objective are given in Lemma \ref{t:objectivesingle}. The constraints \eqref{e:gainStructsingle} can be found in \cite{ref:ChaHokLyg-11}. The proposed control input \eqref{e:policy} satisfies hard input constraint \eqref{e:controlset} as long as the following condition is satisfied: $ \norm{\offset_t + \gain_t \ee(\wnoise_{t:N-1})}_{\infty} \leq U_{\max} $  for all $\ee(\wnoise_{t:N-1})$ such that $\norm{\ee(\wnoise_{t:N-1})}_{\infty} \leq \varphi_{\max}$. 
This is equivalent to the following condition for all $i = 1,\cdots, Nm $:
\begin{align*}
& \max_{\norm{\ee(\wnoise_{t:N-1})}_{\infty} \leq \varphi_{\max}} \abs{ \offset_t^{(i)} + \gain_t^{(i)}\ee(\wnoise_{t:N-1}) }  \leq U_{\max} \\
& \iff \abs{\offset_t^{(i)}} + \norm{\gain_t^{(i)}}_1 \varphi_{\max} \leq U_{\max}.
\end{align*}
  Hence, the constraint \eqref{e:decisionboundsingle} is equivalent to the hard control constraint \eqref{e:controlset}. 
  \end{proof}

\begin{proof}[Proof of Theorem \ref{th:burst}]
Claims follow from Lemma \ref{t:objective} along the same lines of arguments as in the proof of Theorem \ref{th:seq}. 
\end{proof} 	
\begin{proof}[Proof of Theorem \ref{th:rep}]
Claims follow immediately from Lemma \ref{t:objective} with the same lines of arguments as in proof of Theorem \ref{th:seq}. The only difference is that the matrix $\calK$ is replaced by $\calG$ in the present case.
\end{proof}	

\begin{proof}[Proof of Theorem \ref{th:stability}]
\secondrevised{
The constraints \eqref{e:decisonConstraint1single} and \eqref{e:decisonConstraint2single} are obviously convex. 
Now, we show that there exists a sequence of input vectors that is feasible with respect to the input constraint set \eqref{e:controlset}, and satisfies \eqref{e:decisonConstraint1single} and \eqref{e:decisonConstraint2single}.
Consider $\zeta \in \left]0, \frac{U_{\max}}{\sqrt{d_o}\sigma_{1}(\reachab_{\reachindex}(\Aortho,\Bortho)^{+})} \right[ $, $(\offset_t)_{1:\reachindex m} = - \reachab_{\reachindex}(\Aortho, \Bortho)^+ \Aortho^{\reachindex (t+1)} \sat_{r, \zeta}^\infty \bigl((\Aortho\transp)^{\reachindex t} \stortho_{\reachindex t}\bigr), (\gain_t)_{1:\reachindex m} = 0$, and construct $\reachindex$ blocks of control in form of \eqref{decision}. This control is encompassed within the general policy structure \eqref{e:policy}. Moreover, the controls are bounded by $U_{\max}$ by construction, because for $t \in \Nz$ and $\ell \in \{0,1, \cdots, \reachindex -1 \}$,
$\norm{\control_{t + \ell}} \leq \norm{\control_{t : \reachindex}} \leq \sigma_1(\reachab_{\reachindex}(\Aortho, \Bortho)^+)\sqrt{d_o}\zeta \leq U_{\max}$.
Substituting the above input into the left hand side of \eqref{e:decisonConstraint1single}, for every $j = 1, \ldots, d_o$, we get
$ \left( (\Aortho^{\reachindex})\transp \reachab_{\reachindex}(\Aortho, \Bortho)(\offset_t)_{1:\reachindex m} \right)^{(j)} \leq -\zeta$,
whenever $\left( \stortho_{t} \right)^{(j)} \geq r + \epsilon$. Similarly, the above control sequence satisfies \eqref{e:decisonConstraint2single} whenever $\left( \stortho_{t} \right)^{(j)} \leq -r - \epsilon$.  
Therefore, we have shown that there exists a sequence of input vectors that is feasible with respect to the input constraint set \eqref{e:controlset} and satisfies stability constraints \eqref{e:decisonConstraint1single} and \eqref{e:decisonConstraint2single}.\\ 
It is shown in \cite{hokayem2009stochastic} that the Schur stable system, under bounded control inputs, is mean square bounded. Hence, the claim of Theorem \ref{th:stability} is directly implied by Lemma \ref{t:msbsingle} and Lemma \ref{t:msb}. In particular, the mean square boundedness of the closed-loop states under \ref{a:seq} and \ref{a:burst} is implied by Lemma \ref{t:msbsingle} and \ref{t:msb}, respectively. The mean square boundedness under \ref{a:repetitive} is along the same line of arguments as in Lemma \ref{t:msb}. 
}
\end{proof}
\bibliographystyle{IEEEtran}
\bibliography{refs}

\begin{thebibliography}{10}
\providecommand{\url}[1]{#1}
\csname url@samestyle\endcsname
\providecommand{\newblock}{\relax}
\providecommand{\bibinfo}[2]{#2}
\providecommand{\BIBentrySTDinterwordspacing}{\spaceskip=0pt\relax}
\providecommand{\BIBentryALTinterwordstretchfactor}{4}
\providecommand{\BIBentryALTinterwordspacing}{\spaceskip=\fontdimen2\font plus
\BIBentryALTinterwordstretchfactor\fontdimen3\font minus
  \fontdimen4\font\relax}
\providecommand{\BIBforeignlanguage}[2]{{%
\expandafter\ifx\csname l@#1\endcsname\relax
\typeout{** WARNING: IEEEtran.bst: No hyphenation pattern has been}%
\typeout{** loaded for the language `#1'. Using the pattern for}%
\typeout{** the default language instead.}%
\else
\language=\csname l@#1\endcsname
\fi
#2}}
\providecommand{\BIBdecl}{\relax}
\BIBdecl

\bibitem{ref:May-14}
D.~Q. Mayne, ``Model predictive control: Recent developments and future
  promise,'' \emph{Automatica}, vol.~50, no.~12, pp. 2967 -- 2986, 2014.

\bibitem{mesbah_16_survey}
A.~Mesbah, ``Stochastic model predictive control: An overview and perspectives
  for future research,'' \emph{IEEE Control Systems Magazine}, 2016.

\bibitem{ref:Grune-11}
L.~Gr{\"u}ne and J.~Pannek, \emph{Nonlinear Model Predictive Control}.\hskip
  1em plus 0.5em minus 0.4em\relax Springer, 2011.

\bibitem{ref:rawlings-09}
J.~B. Rawlings and D.~Q. Mayne, \emph{Model Predictive Control: Theory and
  Design}.\hskip 1em plus 0.5em minus 0.4em\relax Nob Hill, Madison, Wisconsin,
  2009.

\bibitem{ref:Bemporad-99}
A.~Bemporad and M.~Morari, ``{Control of systems integrating logic, dynamics,
  and constraints},'' \emph{Automatica}, vol.~35, no.~3, pp. 407--427, 1999.

\bibitem{ref:Rossiter-98}
J.~Rossiter, B.~Kouvaritakis, and M.~Rice, ``{A numerically robust state-space
  approach to stable-predictive control strategies},'' \emph{Automatica},
  vol.~34, no.~1, pp. 65--73, 1998.

\bibitem{ref:Kerrigan-04}
E.~C. Kerrigan and J.~M. Maciejowski, ``{Feedback min-max model predictive
  control using a single linear program: robust stability and the explicit
  solution},'' \emph{International Journal of Robust and Nonlinear Control},
  vol.~14, no.~4, pp. 395--413, 2004.

\bibitem{ref:Maciejowski-09}
E.~Hartley and J.~Maciejowski, ``Initial tuning of predictive controllers by
  reverse engineering,'' in \emph{European Control Conference (ECC), 2009}, Aug
  2009, pp. 725--730.

\bibitem{ref:Marruedo-02}
D.~Marruedo, T.~Alamo, and E.~Camacho, ``{Input-to-state stable {MPC} for
  constrained discrete-time nonlinear systems with bounded additive
  uncertainties},'' in \emph{Proceedings of the 41st IEEE Conf. on Decision and
  Control, 2002.}, vol.~4, 2002.

\bibitem{ref:Bayer-13}
F.~Bayer, B.~Mathias, and F.~Allgower, ``{Discrete-time Incremental ISS: A
  Framework for Robust {NMPC}},'' \emph{European Control Conference}, pp.
  2068--2073, 2013.

\bibitem{ref:rakovic-05}
S.~V. Rakovic, E.~C. Kerrigan, K.~I. Kouramas, and D.~Q. Mayne, ``Invariant
  approximations of the minimal robust positively invariant set,'' \emph{IEEE
  Trans. on Auto. Control}, vol.~50, no.~3, pp. 406--410, 2005.

\bibitem{ref:rakovic-12}
S.~V. Rakovic, B.~Kouvaritakis, M.~Cannon, C.~Panos, and R.~Findeisen,
  ``Parameterized tube model predictive control,'' \emph{IEEE Trans. on Auto.
  Control}, vol.~57, no.~11, pp. 2746--2761, 2012.

\bibitem{ref:maeder-09}
U.~Maeder, F.~Borrelli, and M.~Morari, ``Linear offset-free model predictive
  control,'' \emph{Automatica}, vol.~45, no.~10, pp. 2214--2222, 2009.

\bibitem{ref:quevedo-04}
D.~E. Quevedo, G.~C. Goodwin, and J.~A. De~Dona, ``Finite constraint set
  receding horizon quadratic control,'' \emph{International Journal of Robust
  and Nonlinear Control}, vol.~14, no.~4, pp. 355--377, 2004.

\bibitem{li2014cloud}
Z.~Li, I.~Kolmanovsky, E.~Atkins, J.~Lu, D.~Filev, and J.~Michelini, ``Cloud
  aided semi-active suspension control,'' in \emph{IEEE Symposium on
  Computational Intelligence in Vehicles and Transportation Systems (CIVTS),
  2014}.\hskip 1em plus 0.5em minus 0.4em\relax IEEE, 2014, pp. 76--83.

\bibitem{li2015h}
Z.~Li, I.~Kolmanovsky, E.~Atkins, J.~Lu, and D.~Filev, ``{H}$_{\infty}$
  filtering for cloud-aided semi-active suspension with delayed road
  information,'' \emph{IFAC-PapersOnLine}, vol.~48, no.~12, pp. 275--280, 2015.

\bibitem{quevedo_jurado_TAC}
D.~E. Quevedo and I.~Jurado, ``Stability of sequence-based control with random
  delays and dropouts,'' \emph{IEEE Trans. on Auto. Control}, vol.~59, no.~5,
  pp. 1296--1302, 2014.

\bibitem{ref:mayne-00}
D.~Mayne, J.~Rawlings, C.~Rao, and P.~Scokaert, ``{Constrained model predictive
  control: Stability and optimality},'' vol.~36, no.~6, pp. 789--814, 2000.

\bibitem{ref:KeeGil-88}
S.~S. Keerthi and E.~G. Gilbert, ``Optimal infinite-horizon feedback laws for a
  general class of constrained discrete-time systems: stability and
  moving-horizon approximations,'' \emph{Journal of Optimization Theory and
  Applications}, vol.~57, no.~2, pp. 265--293, 1988.

\bibitem{chatterjee-15}
D.~Chatterjee and J.~Lygeros, ``On stability and performance of stochastic
  predictive control techniques,'' \emph{IEEE Trans. on Auto. Control},
  vol.~60, no.~2, pp. 509--514, Feb 2015.

\bibitem{ref:ChaRamHokLyg-12}
D.~Chatterjee, F.~Ramponi, P.~Hokayem, and J.~Lygeros, ``On mean square
  boundedness of stochastic linear systems with bounded controls,''
  \emph{Systems {\&} Control Letters}, vol.~61, no.~2, pp. 375--380, 2012.

\bibitem{bernardini2010model}
D.~Bernardini, M.~Donkers, A.~Bemporad, and W.~Heemels, ``A model predictive
  control approach for stochastic networked control systems,'' in \emph{2nd
  IFAC Workshop on Estimation and Control of Networked Systems, Annecy,
  France}, vol.~3, 2010, pp. 7--12.

\bibitem{Dolgov2013}
J.~Fischer, M.~Dolgov, and U.~D. Hanebeck, ``On stability of sequence-based
  {LQG} control,'' in \emph{52nd Conf. on Decision and Control (CDC),
  2013}.\hskip 1em plus 0.5em minus 0.4em\relax IEEE, 2013, pp. 6627--6633.

\bibitem{Fischer2013}
J.~Fischer, A.~Hekler, M.~Dolgov, and U.~D. Hanebeck, ``{Optimal Sequence-Based
  LQG Control over TCP-like Networks Subject to Random Transmission Delays and
  Packet Losses},'' \emph{Proceedings of the 2013 American Control Conference
  (ACC 2013)}, pp. 1543--1549, 2013.

\bibitem{Hekler2012}
A.~Hekler, J.~Fischer, and U.~D. Hanebeck, ``{Control over Unreliable Networks
  Based on Control Input Densities},'' \emph{Proceedings of the 15th
  International Conf. on Information Fusion (Fusion 2012)}, pp. 1277--1283,
  2012.

\bibitem{Quevedo2011}
D.~E. Quevedo and D.~Ne\v{s}i\'{c}, ``{Input-to-state stability of packetized
  predictive control over unreliable networks affected by packet-dropouts},''
  \emph{IEEE Trans. on Auto. Control}, vol.~56, no.~2, pp. 370--375, 2011.

\bibitem{Quevedo-12}
------, ``{Robust stability of packetized predictive control of nonlinear
  systems with disturbances and Markovian packet losses},'' \emph{Automatica},
  vol.~48, no.~8, pp. 1803--1811, 2012.

\bibitem{kumar1986stochastic}
P.~R. Kumar and P.~Varaiya, \emph{Stochastic Systems: Estimation,
  Identification and Adaptive Control}.\hskip 1em plus 0.5em minus 0.4em\relax
  Prentice-Hall, Inc., 1986.

\bibitem{goulart-06}
P.~J. Goulart, E.~C. Kerrigan, and J.~M. Maciejowski, ``Optimization over state
  feedback policies for robust control with constraints,'' \emph{Automatica},
  vol.~42, no.~4, pp. 523--533, 2006.

\bibitem{ref:HokChaRamChaLyg-10}
P.~Hokayem, D.~Chatterjee, F.~Ramponi, G.~Chaloulos, and J.~Lygeros, ``Stable
  stochastic receding horizon control of linear systems with bounded control
  inputs,'' in \emph{Proceedings of 19th International Symposium on
  Mathematical Theory of Networks and Systems, Budapest, Hungary}, 2010, pp.
  31--36.

\bibitem{ref:amin-10}
D.~Chatterjee, S.~Amin, P.~Hokayem, J.~Lygeros, and S.~S. Sastry, ``Mean-square
  boundedness of stochastic networked control systems with bounded control
  inputs,'' in \emph{49th IEEE Conf. on Decision and Control (CDC),
  2010}.\hskip 1em plus 0.5em minus 0.4em\relax IEEE, 2010, pp. 4759--4764.

\bibitem{ref:quevedo-15}
D.~E. Quevedo, P.~K. Mishra, R.~Findeisen, and D.~Chatterjee, ``A stochastic
  model predictive controller for systems with unreliable communications,'' in
  \emph{5th {IFAC} Conf. on Nonlinear Model Predictive Control {NMPC} 2015
  Seville, Spain, September 2015}, vol.~48, no.~23, 2015, pp. 57 -- 64.

\bibitem{prabhat2016}
P.~K. Mishra, D.~Chatterjee, and D.~E. Quevedo, ``Stable stochastic predictive
  controller under unreliable up-link,'' in \emph{European Control Conf. (ECC),
  2016}, July 2016, pp. 1444--1449.

\bibitem{bernstein2009matrix}
D.~S. Bernstein, \emph{Matrix Mathematics: Theory, Facts, and Formulas}.\hskip
  1em plus 0.5em minus 0.4em\relax Princeton University Press, 2009.

\bibitem{ref:ChaHokLyg-11}
D.~Chatterjee, P.~Hokayem, and J.~Lygeros, ``Stochastic receding horizon
  control with bounded control inputs---a vector-space approach,'' \emph{IEEE
  Trans. on Auto. Control}, vol.~56, no.~11, pp. 2704--2711, 2011.

\bibitem{ref:quevedo-03}
D.~E. Quevedo, G.~C. Goodwin, and J.~S. Welsh, ``Minimizing down-link traffic
  in networked control systems via optimal control techniques,'' in \emph{42nd
  IEEE Conf. on Decision and Control, 2003.}, vol.~2.\hskip 1em plus 0.5em
  minus 0.4em\relax IEEE, 2003, pp. 1200--1205.

\bibitem{ref:aboudolas-10}
K.~Aboudolas, M.~Papageorgiou, A.~Kouvelas, and E.~Kosmatopoulos, ``A
  rolling-horizon quadratic-programming approach to the signal control problem
  in large-scale congested urban road networks,'' \emph{Transportation Research
  Part C: Emerging Technologies}, vol.~18, no.~5, pp. 680--694, 2010.

\bibitem{ref:garcial-89}
C.~E. Garcia, D.~M. Prett, and M.~Morari, ``Model predictive control: theory
  and practice -- a survey,'' \emph{Automatica}, vol.~25, no.~3, pp. 335--348,
  1989.

\bibitem{ref:lofberg-03}
J.~L{\"o}fberg, ``{Approximations of closed-loop minimax {MPC}},'' in
  \emph{42nd IEEE Conf. on Decision and Control}, vol.~2, 2003, pp. 1438--1442.

\bibitem{cinquemani-11}
E.~Cinquemani, M.~Agarwal, D.~Chatterjee, and J.~Lygeros, ``Convexity and
  convex approximations of discrete-time stochastic control problems with
  constraints,'' \emph{Automatica}, vol.~47, no.~9, pp. 2082--2087, 2011.

\bibitem{ref:Hokayem-12}
P.~Hokayem, E.~Cinquemani, D.~Chatterjee, F.~Ramponi, and J.~Lygeros,
  ``{Stochastic receding horizon control with output feedback and bounded
  controls},'' \emph{Automatica}, vol.~48, no.~1, pp. 77--88, 2012.

\bibitem{hokayem2009stochastic}
P.~Hokayem, D.~Chatterjee, and J.~Lygeros, ``On stochastic receding horizon
  control with bounded control inputs,'' in \emph{Proceedings of the 48th IEEE
  Conf. on Decision and Control, held jointly with the 28th Chinese Control
  Conference}.\hskip 1em plus 0.5em minus 0.4em\relax IEEE, 2009, pp.
  6359--6364.

\bibitem{ref:lofberg-04}
J.~L{\"o}fberg, ``{YALMIP}: A toolbox for modeling and optimization in
  matlab,'' in \emph{International Symposium on Computer Aided Control Systems
  Design, 2004}.\hskip 1em plus 0.5em minus 0.4em\relax IEEE, 2004, pp.
  284--289.

\bibitem{ref:toh-06}
K.~Toh, M.~J. Todd, and R.~H. T{\"u}t{\"u}nc{\"u}, ``On the implementation and
  usage of {SDPT3}--a matlab software package for semidefinite-quadratic-linear
  programming, version 4.0,'' in \emph{Handbook on semidefinite, conic and
  polynomial optimization}.\hskip 1em plus 0.5em minus 0.4em\relax Springer,
  2012, pp. 715--754.

\bibitem{ref:PemRos-99}
R.~Pemantle and J.~S. Rosenthal, ``Moment conditions for a sequence with
  negative drift to be uniformly bounded in {$L^r$},'' \emph{Stochastic
  Processes and their Applications}, vol.~82, no.~1, pp. 143--155, 1999.

\bibitem{meyn2012markov}
S.~P. Meyn and R.~L. Tweedie, \emph{Markov Chains and Stochastic
  Stability}.\hskip 1em plus 0.5em minus 0.4em\relax Springer Science \&
  Business Media, 2012.

\bibitem{ref:robert-13}
C.~Robert and G.~Casella, \emph{Monte Carlo Statistical Methods}.\hskip 1em
  plus 0.5em minus 0.4em\relax Springer Science \& Business Media, 2013.

\bibitem{queahl13a}
D.~E. Quevedo, A.~Ahlen, and K.~H. Johansson, ``State estimation over sensor
  networks with correlated wireless fading channels,'' \emph{IEEE Trans. on
  Auto. Control}, vol.~58, no.~3, pp. 581--593, 2013.

\bibitem{ref:RamChaMilHokLyg-10}
F.~Ramponi, D.~Chatterjee, A.~Milias-Argeitis, P.~Hokayem, and J.~Lygeros,
  ``Attaining mean square boundedness of a marginally stable stochastic linear
  system with a bounded control input,'' \emph{IEEE Trans. on Auto. Control},
  vol.~55, no.~10, pp. 2414--2418, 2010.

\end{thebibliography}

\end{document}